\documentclass[12 pt]{amsart}
\usepackage{amsmath,amsfonts,amssymb,amsthm,ifpdf,mathtools}
\usepackage{graphicx,color}
\ifpdf
\usepackage[pdftex,pdfstartview=FitH,pdfborderstyle={/S/B/W 1}, %
colorlinks=true, linkcolor=blue, urlcolor=blue,citecolor=blue,%
pagebackref=true]{hyperref}
\else
\usepackage[hypertex,colorlinks]{hyperref}
\fi

\usepackage{bookmark}

\DeclareMathRadical{\sqrtsign}{symbols}{"70}{largesymbols}{"70}
\newcommand{\bb}{\mathbb}

\newcommand{\integers}{{\bb Z}}
\newcommand{\natls}{{\bb N}}

\newcommand{\reals}{{\bb R}}

\newlength{\figboxwidth}             
\renewcommand{\bold}[1]{\medskip \noindent {\bf #1 }\nopagebreak}

\newcommand{\isom}{\cong}

\newcommand{\cross}{\times}
\newcommand{\st}{\;\: : \;\:}         %Such that

\newcommand{\zed}{\integers}

\newcommand{\diam}{\operatorname{diam}}

\def\@ifundefined#1#2#3%
  {\expandafter\ifx\csname#1\endcsname\relax#2\else#3\fi}

\@ifundefined{theoremstyle}{
}{
\theoremstyle{plain} %default
}
\newtheorem{theorem}{Theorem}[section]
\newtheorem{prop}[theorem]{Proposition}

\newtheorem{lemma}[theorem]{Lemma}
\newtheorem{cor}[theorem]{Corollary}
\newtheorem{corollary}[theorem]{Corollary}
\newtheorem{claim}[theorem]{Claim}

\@ifundefined{theoremstyle}{
}{
\theoremstyle{definition} %default
}

\newcommand{\cD}{{\mathcal D}}

\newcommand{\cM}{{\mathcal M}}

\newcommand{\cR}{{\mathcal R}}
\newcommand{\cS}{{\mathcal S}}

\mathchardef\GG="321D
\newcommand{\mcc}[1]{{}} 
\newcommand{\mccc}[1]{{}} 

\definecolor{mygreen}{rgb}{0,.5,0}

\numberwithin{equation}{section}

\title{Self-Joinings for 3-IETs}

\author{Jon Chaika}
\thanks{The research of J. Chaika was supported in part by NSF grants DMS-135500 and DMS- 1452762, the Sloan foundation, Poincar\'e chair, Warnock chair}
\address{Department of Math
155 South 1400 East, JWB 233
Salt Lake City, UT 84112}
\email{chaika@math.utah.edu}

\author{Alex Eskin}
\thanks{The Research  of  A. Eskin is partially supported  by the
  Simons Foundation and 
NSF grant DMS 1201422}
\address{
Department of Mathematics,
University of Chicago,
Chicago, Illinois 60637, USA\\
}
\email{eskin@math.uchicago.edu}

\usepackage{enumerate}
\usepackage{enumitem}

\theoremstyle{plain}

\newtheorem{defin}{Defintion}

\theoremstyle{definition}

\newtheorem{rem}{Remark}

\newtheorem{ques}{Question}

\newcommand{\fp}[1]{\langle\langle #1 \rangle\rangle}

\newcommand{\future}[1]{{}}

\begin{document}

\begin{abstract}
We show that typical interval exchange transformations on three
intervals are not 2-simple answering a question of Veech. Moreover, the set of self-joinings of almost every 3-IET is
 a Paulsen simplex. 
\end{abstract}

\maketitle
\section{Introduction}
\begin{defin} Let $(X,\mathcal{B},\mu,T)$ be a probability measure preserving system. A \emph{self-joining} is a $T \times T$ invariant measure on $X \times X$ with marginals $\mu$. 
\end{defin}
\begin{defin}
$(X,\mathcal{B},\mu,T)$  is called  \emph{2-simple} if every ergodic
self-joining, other than $\mu \times \mu$, is one-to-one on almost every fiber.
\end{defin}
%\begin{defin} Let $(X,\mathcal{B},\mu,T)$ be a probability measure preserving system. It is \emph{simple} if every ergodic self joining, other than $\mu \times \mu$ is one-to-one on a.e. fiber.
%\end{defin}
\begin{defin}A \emph{Poulsen simplex} is a metrizable simplex where the extreme points are dense.
\end{defin}
Lindenstrauss, Olsen and Sternfeld proved that a Poulsen simplex is unique up to affine homeomorphism \cite{LinOlsStern}. 

\begin{defin} a \emph{3-interval exchange transformation} is defined by 3 non-negative numbers $\ell_1,\ell_2,\ell_3$. It is $T:[0,\ell_1+\ell_2+\ell_3) \to [0,\ell_1+\ell_2+\ell_3)$ by 
\begin{displaymath}
T(x) = \begin{cases}
x + \ell_2 + \ell_3 & \text{if $x \le \ell_1 + \ell_2$} \\
x + \ell_2 + \ell_3 - (\ell_1 + 2 \ell_2 + \ell_3) & \text{otherwise.}
\end{cases}
\end{displaymath}
\end{defin}

\begin{theorem}
\label{theorem:Veech}
Almost every 3-IET is not 2-simple. Also, its self-joinings form a Poulsen simplex.
\end{theorem}
Note that $T \times T$ has topological entropy 0. %Though there may be earlier examples, we have not found another example of a topological entropy zero system whose invariant probability measures form a Poulsen simplex in the literature. 
\medskip

The first part of Theorem~\ref{theorem:Veech}
answers a question of Veech in the negative
\cite[Question 4.9]{veech crit}. (In \cite{veech crit} ``$2$-simple''
is called ``Property $S$.'') 
%After his paper 
%The proof (modulo the proof of Proposition~\ref{prop:3IET}) is at the
%end of Section \ref{sec:theorem1}.

Recall that a measure preserving system is called {\em prime} if it
has no non-trivial factors. In the paper \cite{veech
  crit} mentioned above, Veech 
classified the factors of  $2$-simple systems,
 and so a natural question remains: 
\begin{ques}
Is almost every 3-IET prime?
\end{ques}
It is also natural to wonder what happens for IETs with other permutations and flows on translation surfaces. It is likely that our techniques can show that residual sets of interval exchange transformations on more intervals, and flows on translation surfaces of genus greater than 1 are not simple, but we do not see how they can be applied to almost every flow on translation surface or IET with different permutation.

To prove Theorem~\ref{theorem:Veech} we define in Section
\ref{sec:strongly rank1} a distiguished class of
self-joinings called ``shifted power joinings.'' In Section
\ref{sec:strongly rank1} we also show that a special type of
transformations called ``rigid rank 1 by intervals'' (which includes IET's by  \cite[Part 1, Theorem 1.4 ]{veech metric})
have the property that linear combinations of shifted power joinings
are dense in their self-joinings.  M. Lemanczyk brought to our attention that this result was proved in an unpublished paper of J. King \cite{king}. 
We then prove that almost every 3-IET has the property that its
ergodic self-joinings  are dense in linear combinations of the shifted power joinings. We do this by having an abstract criterion (Section \ref{sec:KSV}) and showing 3-IETs verify this criterion (Section \ref{sec:theorem1}).

\noindent
\textbf{Context of our results:} Before Veech's work, D. Rudolph introduced the notion of \emph{minimal self joinings}, using it as a fruitful class of examples, including examples of prime systems \cite{min self join}.  The property of 2-simple generalizes minimal self joinings and in particular, 
no rigid system has minimal self joinings. The typical IET is rigid \cite[Part 1, Theorem 1.4]{veech metric}, so the typical IET does not have minimal self joinings, 
 but there are rigid 2 simple systems.
Ageev proved that the set of measure preserving transformations which are not 2 simple contains a dense $G_\delta$, i.e. it is a residual set, (with the topology being the so called \emph{weak topology}) \cite{ageev}.  Our construction can be modified to give a new proof of this fact. 

Our result that the self-joinings form a Paulsen simplex is also perhaps a little unexpected. 
Many examples of systems whose set of invariant measures form a Paulsen simplex are well known, but typically these systems are high complexity, satisfying some form of \emph{specification}. 
In contrast, our examples have very low complexity, as $T\times T$ has \emph{quadratic block growth}. 
Since systems of \emph{linear block growth} have only finitely many ergodic measures \cite{bosh lin}, such a system can not have that the set of its invariant measures form a Paulsen simplex (though as our examples show its Cartesian product could). 
We remark that in the previously mentioned unpublished work, J. King proved a residual set of measure preserving transformations (which therefore must include rank 1 transformations) have that their set of their self-joinings form a Paulsen simplex \cite{king}, giving many (non-explicit) entropy zero examples.
 Our result is perhaps still surprising, because we treat a previously considered family of examples and we show typicality in a metric, rather than topological setting.

 Two key steps are showing that the typical 3-IET admit $(n_j,n_j+1)$ approximation (see the proof of Proposition \ref{prop:two meas}) and that this implies the existence of all sorts of ergodic joinings (see Proposition \ref{prop:KSV:joining}). Some consequences of transformations with $(n_j,n_j+1)$ approximation were studied by Ryzhikov \cite{ryzh3} and as a result we get some spectral consequences for $T^n$ and $\underbrace{T\times...\times T}_{n \text{ times}}$, see Remark \ref{rem:simple}. 
 
 \noindent
 \textbf{Acknowledgments:} We thank El Abdaloui, M. Lemaczyk and V. Ryzhikov for numerous illuminating discussions on connections of other results with this work.

\section{Joinings of rigid rank 1 transformations come from limits of linear combinations of powers}\label{sec:strongly rank1}
Let $([0,1],\mathcal{M},\lambda, T)$ be an ergodic invertible transformation.
\begin{defin}We say $T$ is \emph{rigid rank 1 by intervals} if  there exists a sequence of intervals $I_1,\dots$ and natural numbers $n_1,\dots$ so that 
\begin{itemize}
\item $T^iI_k$ is an interval with $diam(T^iI_k)=diam(I_k)$ for all $0\leq i<n_k$.
\item $T^iI_k\cap T^jI_k=\emptyset$ for all $k$ and $0\leq i<j<n_k$.
\item $\underset{k \to \infty}{\lim} \lambda(\cup_{i=0}^{n_k-1}T^iI_k)=1$.
\item $\underset{k \to \infty}{\lim} \frac{\lambda(T^{n_k}I_k\Delta I_k)}{\lambda(I_k)}=0$. 
\end{itemize}
\end{defin} 
This is a condition saying that our transformation is well
approximated by periodic transformations. A similar condition,
admiting \emph{cyclic approximation by periodic transformations} was
considered in \cite{KS}.

Let
\begin{displaymath}
\cR_k=\bigcup_{i=0}^{n_k-1}T^iI_k,
\end{displaymath}
\begin{equation}
\label{eq:def:hatRk}
\hat{\cR}_k = \bigcup_{i=0}^{n_k-1} T^i( I_k \cap T^{-n_k} I_k \cap
T^{n_k} I_k),
\end{equation}
\begin{displaymath}
\tilde{\cR}_k = \bigcup_{i=0}^{n_k-1} T^i( I_k \cap T^{-n_k} I_k \cap
T^{-2n_k} I_k \cap T^{n_k} I_k \cap T^{2 n_k} I_k),
\end{displaymath}
Then, 
$\cR_k$ is the Rokhlin tower over $I_k$,  $\hat{\cR}_k$ is the Rokhlin
tower over $I_k \cap T^{-n_k}I_k\cap T^{n_k}I_k$, and $\tilde{\cR}_k$ is the Rokhlin tower over 
$\bigcap_{i=-2}^2T^{in_k}I_k$. 
We have
\begin{equation}
\label{eq:property:hatRk}
\hat{\cR}_k=\{x:T^ix \in \cR_k \text{ for all }-n_k< i<n_k\},
\end{equation}
and
\begin{equation}
\label{eq:property:tildeRk}
\tilde{\cR}_k=\{x:T^ix \in \cR_k \text{ for all }-2n_k< i< 2n_k\}. 
\end{equation}
Heuristically one can think of $\cR_k$ as the set of points we can control. $\hat{\cR}_k$ and $\tilde{\cR}_k$ let us control the points for long orbit segments, which is necessary for some of our arguments. 
\begin{lemma}\label{lemma:srank est} $\underset{k \to \infty}{\lim}\lambda(\tilde{\cR}_k)=1=\underset{k \to \infty}{\lim}\lambda(\cR_k)=\underset{k \to \infty}{\lim}\lambda(\hat{\cR}_k)$. 
\end{lemma}
\begin{proof}By the third condition in the definition of rigid rank 1
by intervals we have $\underset{k \to \infty}{\lim}
\lambda(\cR_k)=1$. By (\ref{eq:def:hatRk}),
\begin{displaymath}
\lambda(\hat{\cR}_k)\geq\lambda(\cR_k)-n_k\lambda(I_k\setminus
(T^{n_k}I_k \cup T^{-n_k}I_k))
\geq\lambda(\cR_k)-2n_k\lambda(I_k\setminus T^{n_k}I_k),
\end{displaymath}
and thus by the fourth condition of the definition of rigid rank 1 by
intervals, $\underset{k \to \infty}{\lim}\lambda(\hat{\cR}_k) \to 1$. 
% By the fourth condition we have that 
% $\underset{k \to \infty}{\lim}\, \lambda(R_k \Delta \cup_{i=0}^{n_k-1}T^{i }(I_k \cap T^{n_k}I_k \cap T^{-n_k}I_k))=0 $ and so $\underset{k \to \infty}{\lim} \lambda(\hat{R}_k)=1$ because $\hat{R}_k= \cup_{i=0}^{n_k-1}T^{i }(I_k \cap T^{n_k}I_k \cap T^{-n_k}I_k)$.  
Similarly, $\underset{k \to \infty}{\lim}
\lambda(\tilde{R}_k)=1$. 
\end{proof}

%\bold{Shifted power joinings.}
%Let $(X,T,\mu)$ and $(Y,S,\nu)$ be measure preserving dynamical
%systems. Recall that a joining of $(X,T,\mu)$ and $(Y,T,\nu)$ is a
%$T \cross S$-invariant measure $\sigma$ on $X \cross Y$ which projects
%to $\mu$ under the map $X \cross Y \to X$ and to $\nu$ under the map
%$X \cross Y \to Y$. 

\begin{defin}[Shifted Power Joining]
Let $(X,T,\mu)$ be a measure preserving dynamical
system. A self-joining of $(X,T,\mu)$ %with $(X,T,\mu)$
that gives full measure to $\{(x,T^a x)\}$ for some $a\in \zed$ with
$a \ne 0$ is called a \emph{shifted power joining}. 
\end{defin}
These have also been called \emph{off diagonal joinings}.

Let $\iota:[0,1] \to [0,1]$ by $x \to (x,x)$. Let
$\mu=\iota_*\lambda$. Shifted power joinings have the form $(id \times
T^a)_* \mu$ for some $a \in \mathbb{Z} \setminus \{0\}$.

\bold{The operator $A_\sigma$ and convergence in the strong operator topology.}
Let $\sigma$ be a self-joining of $(T,\lambda)$. 
 Let $\sigma_x$ be the corresponding measure on $[0,1]$ coming from disintegrating along $\sigma$ on the fiber $\{x\}\times [0,1]$. 
Define $A_\sigma: L^2(\lambda) \to L^2(\lambda)$ by $A_\sigma(f)[x]=\int f
d\sigma_x$.

Recall that one calls the \emph{strong operator topology}
the topology of pointwise convergence on $L^2(\lambda)$. That is
$A_1,...$ converges to $A_\infty$ in the strong operator topology if
and only if $\underset{ i \to \infty}{\lim}\|A_if-A_{\infty}f\|_2=0$ for all $f \in L^2(\lambda)$. 
\begin{theorem} 
\label{theorem:SOT close}
Assume $([0,1],T,\lambda)$ is rigid
rank 1 by intervals and $\sigma$ is a self-joining of
$([0,1],T,\lambda)$. Then $A_\sigma$ is the strong operator topology
(SOT) limit of linear combinations, with non-negative coefficients, of powers of $U_T$, where $U_T:
L^2([0,1],\lambda) \to L^2([0,1],\lambda)$ denotes the Koopman
operator $U_T(f) = f \circ T$. 
\end{theorem}
\begin{cor}\label{cor:WOT close} (J. King) Any self-joining of a rigid rank 1 by intervals transformation is a weak-* limit of linear combinations of shifted power joinings.
\end{cor}
These results (or very closely related results) were established earlier by J. King \cite{king} using a different proof. In fact he shows that if the joining in Corollary \ref{cor:WOT close} is ergodic then there is no need to take a linear combination. See also \cite[Theorem 7.1]{JRR}. There is an open question of whether this result is true for general rank 1 systems \cite[Page 382]{King r1}. Ryzhikov has a series of results in this direction, see for example \cite{ryzh1} and \cite{ryzh2}. 
%Note that these are a partial answer to a much harder question of King 
%who asked if for rank 1 transformations shifted power joinings were dense in ergodic joinings \cite[Page 382]{King r1}.
\subsection{Proof of Theorem \ref{theorem:SOT close}}
\begin{lemma}
For each $0\leq j<n_k$ we have 
\begin{equation}
\label{eq:lemma1:4}
n_k\int_{T^jI_k}\sigma_x (\cR_k^c)d \lambda(x)\leq \lambda (\tilde{\cR}_k^c).
\end{equation}
\end{lemma}
\bold{Remark.} Note that $n_k$ is roughly $\lambda(T^jI_k)^{-1}$. 
\begin{proof}
Suppose $0 \leq j < n_k$, and suppose $x \in T^j I_k$. 
From (\ref{eq:property:tildeRk}) we have
$T^{i}\cR_k^c \subset \tilde{\cR}_k^c$ for all $-n_k<i<n_k$. 
We claim that
\begin{equation}
\label{eq:sigmax:Rkc}
\sigma_x (\cR_k^c)\leq \sigma_{T^{\ell} x}(\tilde{\cR}_k^c) \qquad \text{for all $-n_k < \ell<n_k$.} 
\end{equation}
Indeed,
$\sigma_{x}(\cR^c_k)=\sigma_{T^\ell x}(T^{\ell}\cR^c_k)\leq
\sigma_{T^\ell x}(\tilde{\cR}_k^c)$, proving (\ref{eq:sigmax:Rkc}). 
Integrating (\ref{eq:sigmax:Rkc}) we get
\begin{equation}
\label{eq:int:TjIk:sigma:y}
\int_{T^jI_k}\sigma_y(\cR_k^c)d\lambda(y)\leq
\int_{T^{j+\ell}I_k}\sigma_z(\tilde{\cR}_k^c) d\lambda(z) \quad\text{for all
$-n_k <\ell<n_k$.}
\end{equation}
Since we can choose $\ell$ in (\ref{eq:int:TjIk:sigma:y}) so that
$j+\ell$ takes any value in $[0,n_k-1]\cap \zed$, we get
\begin{equation}
\label{eq:int:Tj:Ik:min}
\int_{T^jI_k}\sigma_y(\cR_k^c)d\lambda(y)\leq \min_{0 \le i < n_k}
\int_{T^i I_k}\sigma_z(\tilde{\cR}_k^c) d\lambda(z).
\end{equation}
Now
$$\sum_{i=0}^{n_k-1}\int_{T^iI_k}\sigma_y(\tilde{\cR}_k^c)d\lambda(y)\leq
\int_{[0,1]}\sigma_y(\tilde{\cR}_k^c)d\lambda(y)\leq
\lambda(\tilde{\cR}_k^c),$$ 
where the last estimate
uses that $\sigma$ has projections $\lambda$. So we obtain 
\begin{equation}
\label{eq:min:small}
\min\limits_{0\leq i<n_k} \int_{T^iI_k}\sigma_x
(\tilde{\cR}_k^c)d\lambda(x)\leq  \frac 1 {n_k} \lambda(\tilde{\cR}_k^c).
\end{equation}
Now the estimate (\ref{eq:lemma1:4}) follows from
(\ref{eq:int:Tj:Ik:min}) and (\ref{eq:min:small}). 
\end{proof} 

We want to guess coefficients $c_j$ so that $\sigma$ is close to $\sum_{j=0}^{n_k-1}c_j (id\times T^i)_*\mu$. The next lemma comes up with a candidate pointwise version. Theorem \ref{theorem:SOT close} and Corollary \ref{cor:WOT close} follow because by Egoroff's theorem this choice is almost constant on most of the $T^\ell I_k$ 
and the lemma after this (Lemma~\ref{lemma:other indices}), which shows that they are almost $T$ invariant. 
\begin{lemma}\label{lemma:A close} Let $x \in \hat{\cR}_k \cap T^jI_k$ where $0\leq j<n_k$.  Define $c_i(x)=\sigma_x(T^{a}I_k\cap \hat{\cR}_k)$ where $0\leq a<n_k$ and $i+j \equiv a \,( \text{mod }n_k)$.
 For all 1-Lipschitz $f$ we have 
 $$\left|A_\sigma f(x)-\sum_{i=0}^{n_k-1}c_i(x)f(T^ix)\right|\leq \diam(I_k)+2\|f\|_{\sup}\sigma_x(\hat{\cR}_k^c).$$
\end{lemma}
 Morally $c_j(x)$ is the $\sigma_x$ measure of the level in $R_k$ that is $j$ levels above the level $x$ is on. Because $j+\ell$ can be bigger than $n_k$ the definition is slightly more complicated.  Note that the $c_j(x)$ are non-negative. 
\begin{proof}
Suppose $x \in \hat{\cR}_k \cap T^jI_k$.
First notice that if $y,z \in T^iI_k$ for some $0\leq i<n_k$ we have that $d(y,z)<\diam(I_k)$. So if $j+\ell< n_k$ we have
\begin{equation}
\label{eq:Ac 1}
\left|\int_{\hat{R}_k \cap T^{j+\ell}I_k} f d\sigma_x-
c_j(x)f(T^ix)\right|\leq \|f\|_{Lip}\, \diam(I_k).
\end{equation}
If $j+\ell\geq n_k$ then $|c_j(x)-\lambda(\hat{R}_k\cap
T^{j+\ell}I_{k})|\leq \lambda(\tilde{R}^c_k \cap T^{j +\ell}I_k)$ because if $y \in \tilde{R}_k$ then $T^{\pm n_k}y \in \hat{R}_k$. 
So for any $j$ we have 
\begin{equation}
\label{eq:Ac 2}
\left|\int_{\hat{R}_k \cap T^{j+\ell}I_k} f
d\sigma_x-c_j(x)f(T^jx)\right|\leq \|f\|_{Lip}\,
\diam(I_k)+\|f\|_{\sup}\, \lambda(\tilde{R}_k^c \cap T^{j+\ell}I_k).
\end{equation}
By (\ref{eq:property:hatRk}), for all $0 \leq \ell < n_k$, $T^{-\ell}
\hat{R}_k \subset R_k$. Therefore, $\hat{R}_k \subset
\bigcup_{i=\ell}^{\ell+n_k-1}T^iI_k$ for all $0\leq \ell<n_k$. By
summing over the $j$ in (\ref{eq:Ac 2}) we obtain 
\begin{equation}
\label{eq:Ac 3}
\left|\int_{\hat{R}_k}fd
\sigma_x-\sum_{j=0}^{n_k-1}c_i(x)f(T^ix)\right|\leq \|f\|_{Lip}\, \diam(I_k)+\|f\|_{\sup}\lambda(\tilde{R}_k^c). 
\end{equation}
In view of the fact that 
\begin{equation}
\label{eq:Ac 4}
\left|\int_{\hat{R}_k^c} f\, d\sigma_x\right|\leq \|f\|_{\sup}\, \lambda(\tilde{R}^c_k), 
\end{equation}
we obtain the lemma.
\end{proof}
\begin{lemma}
\label{lemma:other indices}
Suppose $0 \le \ell < n_k$. 
If $x\in T^\ell I_k$ and $-\ell\leq i<n_k-\ell$ then 
$$\sum_{j=0}^{n_k-1}|c_j(x)-c_j(T^ix)|\leq 2\sigma_x(\tilde{R}_k^c).$$
\end{lemma}
\begin{proof}
Suppose $0 \le \ell < n_k$, $0 \leq j < n_k$, and $-\ell \le i <
n_k-\ell$.  
First note that  if $0 \le m < n_k$ and $z \in T^m I_k \cap \hat{R}_k$ then by
(\ref{eq:def:hatRk}), we have $T^s z \in T^{m+s}I_k \cap
\hat{R}_k$ for all $-m\leq s<n_k-m$. Thus, if $j + \ell <
n_k$ and $i+j+\ell < n_k$, we have
$$\sigma_{T^ix}(T^{i+j+\ell}I_k \cap \hat{R}_k)=\sigma_{x}(T^{j+\ell}I_k \cap T^{-i}\hat{R}_k)=\sigma_x(T^{j+\ell}I_k \cap \hat{R}_k).$$
This gives $c_j(x)=c_j(T^ix)$ if $j+\ell <n_k$ and $i+j+\ell<n_k$. 
By similar reasoning we have that $c_j(x)=c_j(T^ix)$ if $j+\ell\ge n_k$ and $i+j+\ell\geq n_k$. 

Now lets assume that $j+\ell<n_k$ and $i+j+\ell\geq n_k$. Then,
\begin{equation}
\label{eq:cjTix}
c_{j}(T^ix)=\sigma_{T^ix}(T^{i+j+\ell-n_k}I_k \cap
\hat{R}_k)=\sigma_x(T^{j+\ell-n_k}I_k\cap T^{-i}\hat{R}_k).
\end{equation}
Also,
\begin{equation}
\label{eq:cjx}
c_j(x)=\sigma_x(T^{j+\ell}I_k \cap \hat{R}_k).
\end{equation}
Now because $\tilde{R}_k\subset \bigcap_{i=-{n_k}}^{n_k}T^i\hat{R}_k $, if $z \in T^{i+j+\ell-n_k} I_k \cap \tilde{R}_k$,  then,
$z \in T^{j+\ell-n_k}I_k\cap T^{-i}\hat{R}_k$, and $z \in
T^{j+\ell}I_k \cap \hat{R}_k$. Therefore, the symmetric difference
between $T^{j+\ell-n_k}I_k\cap T^{-i}\hat{R}_k$ and $T^{j+\ell}I_k
\cap \hat{R}_k$ is contained in the union of $T^{i+j+\ell-n_k} I_k \cap
\tilde{R}_k^c$ and $T^{j+\ell} I_k \cap \tilde{R}_k^c$. 
Thus, in view of (\ref{eq:cjTix}), and (\ref{eq:cjx}), 
\begin{displaymath}
|c_j(x)-c_j(T^ix)|\leq
\sigma_x(T^{j+\ell+i-n_k}\tilde{R}_k^c)+
\sigma_x(T^{j+\ell}\tilde{R}_k^c).
\end{displaymath}
 The last case, where $j+\ell \geq n_k$ and $0\leq i+j+\ell<n_k$ gives analogous bounds. So we bound $\sum_{i=0}^{n_k-1} |c_j(x)-c_j(T^ix)|$ by $2\sum_{i=0}^{n_k-1}\lambda(T^iI_k\cap \tilde{R}_k^c)\leq 2\lambda( \tilde{R}_k^c)$ and obtain the lemma.
\end{proof}

Let $d_{KR}$ denote the Kantorovich-Rubinstein metric on
measures. That is
\begin{displaymath}
d_{KR}(\mu,\nu)=\sup \left\{\left|\int fd\mu-\int f d\nu\right|:f \text{ is 1-Lipschitz}\right\}.
\end{displaymath}
The next lemma is an immediate consequence of this definition.
\begin{lemma}\label{lemma:kr est}If $f$ is 1-Lipshitz and
$d_{KR}(\sigma_x,\sigma_y)<\epsilon$ then $|A_\sigma f(x)-A_\sigma
f(y)|<\epsilon$.
\end{lemma}

We say $0 \le j<n_k$ is \emph{k-good} if there exists $y_j$ in $T^jI_k$ so
that at least $1-\epsilon$ proportion of the points in $T^jI_k$ have their
disintegration $\epsilon$ close to $y_j$. That is
\begin{displaymath}
\lambda(\{x \in T^jI_k \st
  d_{KR}(\sigma_x,\sigma_{y_j})<\epsilon\}) \ge (1-\epsilon) \lambda(I_k).
\end{displaymath}

\begin{lemma}
\label{lemma:most:good}
For all $\epsilon>0$ there exists
$k_0$ so that for all $k>k_0$ we have
$$|\{0 \le j < n_k : j \text{ is k-good }\}|>(1-\epsilon)n_k.$$
\end{lemma}
\begin{proof}
By Lusin's Theorem there exists a compact set $K$ of measure at least
$1-\frac {\epsilon^2} 4$ so that the map $y \to \sigma_y$ is
continuous with respect to the usual metric on $[0,1]$ and the metric
$d_{KR}$ on measures. Because $K$ is compact this map is uniformly
continuous and so there exists $\delta>0$ so that  $x,y \in K$ and
$|x-y|<\delta$ then $d_{KR}(\sigma_x,\sigma_y)<\epsilon$. We choose
$k$ so that $\diam(I_k)<\delta$ and $\lambda([0,1] \setminus
\cR_k)<\frac {\epsilon^2} 4$.  
Let
\begin{displaymath}
\eta = \frac{1}{n_k}|\{0\leq j<n_k\st \lambda(T^j I_k\cap K^c)> \epsilon
\lambda(I_k)\}|. 
\end{displaymath}
Then, because the $T^j I_k$ are disjoint and of equal size and
$\bigcup_{j=0}^{n_k-1} T^j I_k = \cR_k$, it is clear that  
\begin{displaymath}
\eta \epsilon \le \frac{\lambda(K^c \cap\cR_k)}{\lambda(\cR_k)} \le \frac{\frac{\epsilon^2}4}{1-\frac {\epsilon^2}4}<\frac{\epsilon^2}2,
\end{displaymath}
and thus $\eta < \epsilon/2$. This completes the proof of the lemma.
\end{proof}

\bold{Notation.}
If $j$ is $k$-good let 
$$G_j=\{x \in T^jI_k: \lambda(\{y \in
T^jI_k:d_{KR}(\sigma_x,\sigma_y)<2\epsilon\})>(1-\epsilon)\lambda(I_k)\},$$
i.e.\ $G_j$ is the set of points
that are almost continuity points of the map $x \to \sigma_x$
(restricted to $T^jI_k$). We set $G_j=\emptyset$ if $j$ is not $k$-good. 

\begin{lemma}
\label{lemma:our point}
For all $\epsilon>0$ there exists $k_1$ so that for all $k>k_1$ there
exists $0\leq \ell<n_k$ and $y_k\in T^\ell I_k\cap \hat{\cR}_k$ so
that $\sigma_{y_k}(\tilde{\cR}_k^c)<\epsilon$ and
\begin{equation}
\label{eq:good:yk}
|\{-\ell\leq j<n_k-\ell:T^j y_k \in G_{\ell+j} \text{ and }j \text{ is $k$-good}\}|>(1-12\epsilon)n_k. 
\end{equation}
\end{lemma}
\begin{proof} If $j$ is $k$-good then 
$$\lambda(G_j)>(1-\epsilon)\lambda(I_k).$$
Let $\cR_k^* = \bigcup_{j=0}^{n_k-1} G_j$. 
Notice that  $\underset{k\to \infty}{\lim} \, \lambda(\cup_{i=0}^{n_k-1}T^iI_k)=\underset{k\to\infty}{\lim}\, \lambda(\cR_k)=1$  and so for all large enough $k$ (so that $\lambda(\cR_k)$ is close to 1 and Lemma~\ref{lemma:most:good} holds) we have 
$$\lambda(\cR_k^*)\geq  (1-\epsilon)^2\lambda(\cR_k)>1-3\epsilon.$$
By a straightforward $L^1$ estimate, we have 
\begin{multline*}
\sum_{\ell=0}^{n_k-1}\lambda(\{y\in T^\ell I_k:|\{-\ell \leq
j<n_k-\ell:G_j = \emptyset \text{ or } T^jy\not\in G_{j+\ell}\}|\geq 12\epsilon n_k\} )<
\frac{ 3 \epsilon}{12}=\frac   \epsilon 4
%< 2\lambda(\cR_k^*)<6 \epsilon.
\end{multline*}
Therefore, the measure of the set of $y_k$ satisfying
(\ref{eq:good:yk}) (for some $\ell$) is at least $1/2$. 

Recalling that by 
Lemma~\ref{lemma:srank est} we have $\underset{k \to \infty}{\lim}\,
\lambda(\tilde{\cR}_k^c)= 0$ and so for $k$ large enough,
$$\lambda(\{y: \sigma_y(\tilde{\cR}_k)>\epsilon\})<\frac 1 3.$$
Thus, we can pick $y_k$ satisfying the conditions of the lemma. 
\end{proof}

\begin{proof}[Proof of Theorem \ref{theorem:SOT close}] 
For each $k$ large enough so that Lemmas \ref{lemma:most:good} and \ref{lemma:our point} hold and $\diam(I_k)<\epsilon$ and $\lambda(\cR_k^c)<\epsilon$, let  $y_{k}$ be as in the statement of Lemma~\ref{lemma:our point} and assume it is in $T^\ell I_k$ for some $0\leq \ell<n_k$.
\medskip

\noindent
\textit{Step 1:} We show that for all 1-Lipschitz functions $f$ with $\|f\|_{\sup}\leq 1$ we have 
$$\underset{k \to \infty}{\lim} \, \|A_\sigma f-\sum_{i=0}^{n_k-1}c_i(y_{k})U_T^if\|_2=0.$$
First, observe that by Lemma~\ref{lemma:A close} and the fact that $\|f\|_{\sup}\leq 1$, 
\begin{multline*}
|A_\sigma f(T^jy_k)-\sum_{i=0}^{n_k-1}c_i(T^jy_k)f(T^{i+j}y_k)|<
\diam(I_k)+2\sigma_{T^jy_k}(\hat{\cR}_k^c)\leq \\ \leq \diam(I_k)+2\sigma_{y_k}(\tilde{\cR}_k^c).
\end{multline*}
 By our assumptions that $\diam(I_k) < \epsilon$ and $\sigma_{y_k}(\tilde{\cR}_k^c)<\epsilon$ we have
 $$ |A_\sigma f(T^jy_k)-\sum_{i=0}^{n_k-1}c_i(T^jy_k)f(T^{i+j}y_k)|<3\epsilon.$$
  From Lemma~\ref{lemma:kr est} we 
have that if $x$ satisfies 
\begin{equation}\label{eq:kr close}d_{KR}(\sigma_x,\sigma_{T^jy_k})<\epsilon \end{equation}
 then 
$$|A_\sigma f(x)-\sum_{i=0}^{n_k-1}c_i(T^jy_k)f(T^{i+j}y_k)|<4\epsilon.$$
Let $V$ denote the set of $x$ satisfying (\ref{eq:kr close}) and such that
$x \in T^{\ell+j}I_k \cap \hat{\cR}_k$ for $-\ell \leq
j<n_k-\ell$. Then, for $x \in V$,  
$T^{i}x,T^{i+j}y_k \in T^{i+\ell+j\, (mod \, n_k)}I_k$ for all $0\leq i<n_k$ since $-n_k<i, \, i+j<n_k$ (by \eqref{eq:property:hatRk}). 
Thus for any $x \in V$,
$$|A_\sigma f(x)-\sum_{i=0}^{n_k-1}c_i(T^jy_k)f(T^ix)|<4\epsilon+\diam(I_k).$$

Recalling that by assumption $\diam(I_k)<\epsilon$ and invoking Lemma
\ref{lemma:other indices} we have 
  $$\int_{V}|A_\sigma f(x)-\sum_{j=0}^{n_k-1}c_j(y_k)f(T^jx)|^2d\lambda(x)\leq (5\epsilon+\sigma_y(\tilde{\cR}_k))^2<(6\epsilon)^2.$$
    Since $y_k$ satisfies the assumptions of Lemma~\ref{lemma:our point} and $\lambda(\tilde{\cR}_k^c)<\epsilon$  we have that 
\begin{equation}\label{eq:V big}\lambda(V^c)<2\epsilon n_k \lambda(I_k)+\epsilon.
\end{equation}
  Estimating trivially on $V^c$ we have
\begin{multline*}
\|A_\sigma f-\sum_{j=0}^{n_k-1}c_j(y_k)f\circ T^j\|_2^2=\int_0^1
  |A_\sigma f(x)-\sum_{j=0}^{n_k-1}c_j(y_k)f(T^jx)|^2 \, d\lambda(x)
  \leq \\ \leq
  (6\epsilon)^2+\|f\|_{\sup}^2\big((2\epsilon n_k)\lambda(I_k)+\epsilon \big).
\end{multline*}
Since $\|f\|_{\sup}\leq 1$ and $\epsilon$ is arbitrary this establishes Step 1.
\medskip

\noindent  
\textit{Step 2:} Completing the proof.

The idea of the proof is that by step 1 and linearity we have the limit on a dense set in $L^2$. Since the functions on $L^2$ we consider have operator norm uniformly bounded (by 1) they are an equicontinuous family and so convergence on a dense set implies convergence. 

To complete the formal proof of the theorem,  observe that for any $z$
we have $\sum c_i(z)=\sum |c_i(z)|\leq \sigma_z([0,1])$ and we may
assume that $\sigma_z([0,1])= 1$.\footnote{It is 1 for all but a
  measure zero
  set of $z$ and we may change the disintegration on this zero set.}
So
\begin{displaymath}
\left\|\sum_{i=0}^{n_k-1}c_i(y_k)U_T^i\right\|_{op}\leq 1 \qquad\text{for all $k$}.
\end{displaymath}
Therefore since we have shown $\underset{k \to \infty}{\lim} \,
\|A_\sigma f- \sum_{i=0}^{n_k-1}c_i(y_k)U_T^if\|_2 =0$ for a set of
$f$ with dense span in $L^2$ (that is 1-Lipschitz functions with
$\|f\|_{\sup}\leq 1$), we know that for all $f \in L^2 $  we have that
$\underset{k \to \infty}{\lim} \, \|A_\sigma f- \sum_{i=0}^{n_k-1}c_i(y_k)U_T^if\|_2 =0.$ This is the definition of strong operator convergence. 
\end{proof}

\begin{proof}[Proof of Corollary \ref{cor:WOT close}]
Let $\hat{\delta}_p$ denote the point mass at $p$.
By the proof of
the theorem that there exists $y_k$ so that
\begin{displaymath}
d_{KR}(\sigma_x,\sum_{j=0}^{n_k-1}c_j(y_k)\hat{\delta}_{(x,T^ix)})<5\epsilon
\end{displaymath}
for all $x \in V$.   By
(\ref{eq:V big}) we may assume $\lambda(V^c)$ is as small as we want. The corollary follows.
\end{proof}

\section{An abstract criterion}\label{sec:KSV}
Let $(S,Y,\lambda)$ be a uniquely ergodic topological dynamical
system. Let $\hat{\delta}_p$ denote a point mass at $p$.  Note we will
consider the metric $d_{KR}$ on the Borel probability measures  on
$Y\times Y$ (which is a weak-* closed set since $Y$ is compact) and
the measures $\hat{\delta}_p$ for $p \in Y \times Y$.  If
$\mu$ is a measure on $Y \times Y$, let $(\mu)_x$ be the
disintegration of $\mu$ along $\{x\}\times Y$.

Motivated by Corollary \ref{cor:WOT close} we wish to build \emph{ergodic} joinings that are close to finite linear combinations of shifted power joinings. For example we wish to have ergodic measures with $d_{KR}$ distance $\epsilon$ from the joining that gives measure $\frac 1 2$ to $\{(x,x)\}$ and measure $\frac 1 2 $ to $\{(x,Sx)\}$. Naively, one wants to find a sequence of shifted power joinings that spend half their time close to $\{(x,x)\}$ and half their time shadowing $\{(x,Sx)\}$. Taking a weak-* limit of these we wish to have a measure close to the joining that gives measure $\frac 1 2$ to $\{(x,x)\}$ and measure $\frac 1 2 $ to $\{(x,Sx)\}$.

  Our approach will be to do this inductively, to have a sequence of measure $\nu_i$ and $\mu_i$ so that $\nu_0$ is the shifted power joining supported on $\{(x,x)\}$ and $\mu_0$ is the joining supported on $\{(x,Sx)\}$. 
Inductively, $\mu_{i+1}$ spends a definite proportion of its time near $\mu_i$ and a definite proportion near $\nu_i$ and similarly for $\nu_{i+1}$. That is, we want to have sets $A_{i+1}$ and $B_{i+1}$ so that when $x \in A_{i+1}$ we have $(\nu_{i+1})_x$ is close to $(\mu_i)_x$ and $(\mu_{i+1})_x$ is close to $(\nu_i)_x$ and when $x \in B_{i+1}$ we have 
$(\nu_{i+1})_x$ is close to $(\nu_i)_x$ and $(\mu_{i+1})_x$ is close
to $(\mu_i)_x$. Clearly we want the union of $A_i$ and $B_i$ to have almost full measure and it is helpful that they each have measure at least $c>0$. This isn't quite good enough, in particular if $A_i$ and $B_i$ were constant sequences. We now make the next technical proposition to overcome these issues and additionally guarantee that limiting joining is ergodic.

Of course we want to consider the case of a linear combination of $d$ off diagonal joinings. That is, if we are given a finite number of shifted power joinings $\nu_0^{(1)},...,\nu_0^{(d)}$  we wish to approximate $\frac 1 d \sum_{i=1}^d\nu_0^{(i)}$. We do this analogously to the previous case. Indeed, we have $A_1,B_1$ and $\{\nu_1^{(i)}\}_{i=1}^d$ so $(\nu_1^{(i)})_x$ is close to $(\nu_0^{(i-1)})_x$ for $x \in A_1$ (where $i-1$ is interpreted as $d$ if $i=1$) and $(\nu_0^{(i)})_x$ for $x \in B_1$. We repeat this and obtain $\{\nu_2^{(i)}\}_{i=1}^d$, $A_2$ and $B_2$. Now $(\nu_2^{(i)})_x$ is close to $(\nu_0^{(i-2)})_x$ for $x \in A_1\cap A_2$. We continue repeating to approximate $\frac 1 d \sum_{i=1}^d\nu_0^{(i)}$.

Proposition~\ref{prop:KSV:joining} makes this precise.  Conditions \ref{cond:defin}-\ref{cond:small error} are basic setup, Condition \ref{cond:track} gives the inductive switching as above and Condition \ref{cond:kr close} lets us rule out a previously mentioned issue to show that the weak-* limit of the $\nu_i$ and $\mu_i$ is close to $\frac 12 (\mu_0+ \nu_0)$ and moreover that it is ergodic. 

%Let $d_{KR}$ denote the Kantorovich-Rubistein metric: 
%$$d_{KR}(\mu,\nu)=\sup\{|\int fd\mu-\int fd\nu|:f \text{ is 1-Lipschitz and  }\|f\|_{\sup}\leq 1\}.$$

Let $J_k$ be a sequence of intervals, $U_k$ be a sequence of measurable sets, $r_k$ be a sequence of natural numbers, $n_k^{(\ell)}$ be sequences of natural numbers for $\ell \in \{1,...,d\}$ and $\epsilon_j>0$ be a sequence of real numbers. 
Let $A_k=\bigcup_{i=1}^{r_k}S^i(J_k) \setminus U_k$ and $B_k=A_k^c
\setminus U_k$.  Let $\nu_{k}^{(\ell)}$ be the unique $S\times S$
invariant probability measure supported on $\{(x,S^{n_k^{(\ell)}}x)\}$. Note
that the system $(Y\cross Y,S\cross S,\nu_k^{(\ell)})$ is isomorphic
to $(S,Y,\lambda)$.  
Note that  $(\nu_j^{(\ell)})_x$, %the disintegration of $\nu_j^{(\ell)}$
 %along $\{x\} \times Y$, 
  is a point mass at $(x,S^{n_j^{(\ell)}}x)$. 
\begin{prop}
\label{prop:KSV:joining}
Assume 
\begin{enumerate}[label=(\alph*)]
\item\label{cond:defin} There exists $c>0$ so that for all $k$ we have
$\lambda(A_k) > c$ and $\lambda(B_k)>c$. 
\item \label{con:return}The minimal return time of $S$ to $J_k$ is at least $\frac 3 2 r_k.$
\item\label{cond:U small} $\lambda(U_k)<\epsilon_k$.
\item \label{cond:scale grow}$\underset{ k \to
  \infty}{\lim} r_k \sum_{i>k}\lambda(J_i)=0$. 
\item \label{cond:small error} $\epsilon_i$ are non-increasing and $\sum \epsilon_j<\infty$.
\end{enumerate}
If 
\begin{enumerate}[label=(\Alph*)]
\item\label{cond:track} For any $x \in A_k$ we have $d_{KR}((\nu_k^{(\ell)})_x,(\nu_{k-1}^{(\ell-1)})_x)<\epsilon_k$ and for any $x \in B_k$ we have $d_{KR}((\nu_k^{(\ell)})_x,(\nu_{k-1}^{(\ell)})_x)<\epsilon_k$. 

Note $\nu_{k-1}^{(\ell-1)}$ is interpreted to be $\nu_{k-1}^{(d)}$ if $\ell=1$. 
\item\label{cond:kr close}$d_{KR}( \frac 1 L \sum_{i=1}^{L}(S \times
S)^i (\nu_k^{(\ell)})_x,\nu_{k}^{(\ell)})<\epsilon_k$ for all $x \in X$, 
all $L\geq \frac{r_{k+1}}9$ and any $\ell \in \{1,...,d\}$. 
\footnote{Note that since $S \times S$ on $\{(x,S^{n_j^{(\ell)}}x)\}$ is uniquely ergodic, such an $r_{k+1}$ always exists \cite[Proposition 4.7.1]{BS}.} 
\end{enumerate}
Then the weak-* limit of any $\nu_k^{(\ell)} $ (as $k$ goes to infinity) is the same as the weak star limit of $ \frac 1 d \sum_{\ell=1}^d \nu_k^{(\ell)}$ as $k$ goes to infinity. In particular these limits exist. Call this measure $\mu$. It is ergodic and there exists $C$ so that $d_{KR}(\mu, \frac 1 d\sum_{\ell=1}^d \nu_{k}^{(\ell)})\leq C\sum_{j=k}^\infty \epsilon_j$. 
\end{prop}
To connect this to the remarks above, consider the case that the
$\nu_0^{(\ell)}$ are given shifted power joinings and we want an
ergodic measure close to $\frac 1 d \sum \nu_0^{(\ell)}$. Of course
this only treats particular types of linear combinations, but if our
system is rigid (which rigid rank 1 by interval transformations are),
for any shifted power joining we have different shifted power joinings
close to it. For example, if we want to approximate $\tilde{\nu}=\frac 2 3 (T^n\times id)_*\lambda+\frac 1 3 (T^m\times id)_*\lambda$ we choose $k$ so that $T^k \approx id$. This means 
$$\tilde{\nu}\approx \frac 1 3 (T^{n+k}\times id)_*\lambda+\frac 1 3 (T^n \times id)_*\lambda+\frac 1 3 (T^m \times id)_*\lambda$$ and this is the measure we approximate as above. 
This lets us treat general linear combinations of shifted
power joinings. 

\begin{rem}One can drop the assumption that $(S,Y,\lambda)$ is uniquely ergodic. In this case one replaces \ref{cond:kr close} by 
$$\lambda(\{x:d_{KR}( \frac 1 L \sum_{i=1}^{L}(S \times S)^i(\nu_k^{(\ell)})_x,\nu_{k}^{(\ell)})>\epsilon_k \text{ for some }L\geq \frac{r_{k+1}}9\})<\epsilon_{k}.$$
This requires some straightforward changes to the estimates in the proof of Corollary \ref{cor:to:bary} and the definition of the set $G_k$ in the proof of Proposition~\ref{prop:KSV:joining}.
\end{rem}

\subsection{Proof of Proposition~\ref{prop:KSV:joining}}
\begin{lemma}
\label{lemma:to:bary}
Given $c>0$ and $d \in \mathbb{N}$ there exists $\rho<1$, $C$ so that
if $0 < \delta_i <1/2$ and $a_i, b_i$ are such that $a_i,b_i>c$ and $1\geq
a_i+b_i>1-\delta_i$ and also $0 \le \gamma_i^{(\ell)}\le 1$ are
sequences of real numbers for each $\ell \in \{1,...,d\}$ satisfying
\begin{equation}
\label{eq:gammas:assumption}
|\gamma_i^{(\ell)}-(a_i\gamma_{i-1}^{(\ell-1)}+b_i\gamma_{i-1}^{(\ell)})|<\delta_{i-1}
\end{equation}
then 
$$\left|\gamma_i^{(s)}-\frac 1 d \sum_{\ell=1}^d\gamma_k^{(\ell)}\right|\leq
C\sum_{j=k}^{i-1}\left(\delta_i+\frac{\delta_i}{1-\delta_i}\right)+C\rho^{i-k}$$ for
all $k \ge 0$, $i>k$ and $s \in \{1,...,d\}$. 
\end{lemma}
\begin{proof}Let $\hat{\gamma}_k^{(\ell)}=\gamma_k^{(\ell)}$ and inductively let $\hat{\gamma}_i^{(\ell)}=\frac{a_i}{a_i+b_i}\hat{\gamma}_{i-1}^{(\ell-1)}+\frac{b_i}{a_i+b_i}\hat{\gamma}_{i-1}^{(\ell)}$. 
Observe that
% \begin{displaymath}
% |\hat{\gamma}_{k+1}^{(\ell)}-\gamma_k^{(\ell)}| =
% \left|\frac{a_i}{a_i+b_i}\gamma_i^{(\ell-1)}- +\frac{b_i}{a_i+b_i}\gamma_i^{(\ell)}-\gamma_i^{(\ell)}\right|.
% \end{displaymath}
\begin{multline*}
|\hat{\gamma}_{i}^{(\ell)}-\gamma_i^{(\ell)}| \leq
\left|\frac{a_i}{a_i+b_i}\left(\hat{\gamma}_{i-1}^{(\ell-1)} -
\gamma_{i-1}^{(\ell-1)}\right)+\frac{b_i}{a_i+b_i}\left(
\hat{\gamma}_{i-1}^{(\ell)} - \gamma_{i-1}^{(\ell)} \right)\right|+\\
 \left|\frac{a_i}{a_i+b_i}\gamma_{i-1}^{(\ell-1)}+\frac{b_i}{a_i+b_i}\gamma_{i-1}^{(\ell)}-\gamma_i^{(\ell)}\right|.
\end{multline*}
% \begin{multline*}
% |\hat{\gamma}_{i}^{(\ell)}-\gamma_i^{(\ell)}|\leq \left|\hat{\gamma}_i^{(\ell)}-\left(\frac{a_i}{a_i+b_i}\gamma_i^{(\ell-1)}+\frac{b_i}{a_i+b_i}\gamma_i^{(\ell)}\right)\right| +\\
%  \left|\frac{a_i}{a_i+b_i}\gamma_i^{(\ell-1)}+\frac{b_i}{a_i+b_i}\gamma_i^{(\ell)}-\gamma_i^{(\ell)}\right|.
% \end{multline*}
The second term is at most $\frac{\delta_{i-1}}{1-\delta_{i-1}}+\delta_{i-1}$ and using this we inductively see that %the first term is at most
 $|\hat{\gamma}_i^{\ell}-{\gamma}_i^{(\ell)}|\leq \sum_{j=k}^{i-1}\left(\delta_i+\frac{\delta_i}{1-\delta_i}\right)$. 

Thus it suffices to show that there exists $ C,\rho$ so that
${|\hat{\gamma}_i^{(s)}-\frac 1 d
  \sum_{\ell=1}^d\gamma_k^{(\ell)}|<C\rho^{i-k}}$. To see this note
that $\hat{\gamma}_{i+d}^{(s)}=\sum c_{\ell,s}
\hat{\gamma}_i^{(\ell)}$ where $1\ge c_{\ell,s}>\zeta>0$ for some fixed $\zeta$ depending only on
$c$ and $d$. Consider the matrix $A_{i}$ which has $(\ell,s)$ entry
equal to $c_{\ell,s}$. This matrix is a definite contraction in the
Hilbert projective metric. Indeed, for every $\zeta$ there exists $\theta>0$ so that if $M$ is a positive matrix where the ratio of every pair of entries is at most $\zeta$ and $v,w$ are any vectors in the positive cone then  
  $D_{HP}(Mv,Mw)<\theta D_{HP}(v,w)$ where $D_{HP}$ denotes the Hilbert Projective metric. Now
$\hat{\gamma}_{k+rd}^{(\ell)}$ is the $\ell^{\text{th}}$ entry of
$A_{k}A_{k+d}...A_{k+(r-1)d}\tilde{\gamma}$ where $\tilde{\gamma}$ is
the vector whose $i^{\text{th}}$ entry is
$\hat{\gamma}_i^{(k)}$. Since each $A_{i+jd}$ is a definite
contraction in the Hilbert projective metric, we see that
$|\hat{\gamma}_{i+rd}^{(\ell)}-\hat{\gamma}_{i+rd}^{(\ell')}|$ decays
exponentially in $r$. It is straightforward to check that $\frac 1
d\sum_{\ell=1}^d\hat{\gamma}_i^{(\ell)}=\frac 1 d
\sum_{\ell=1}^d\hat{\gamma}_k^{(\ell)}=\frac 1 d
\sum_{\ell=1}^d\gamma_k^{(\ell)}$ and so
$|\hat{\gamma}_{k+rd}^{(\ell)}-\frac 1 d
\sum_{\ell=1}^d\gamma_k^{(\ell)}|$ decays exponentially in $r$. After
choosing $C>\rho^{-d}$ we get $|\hat{\gamma}_{k+j}^{(\ell)}-\frac 1 d \sum_{\ell=1}^d\gamma_k^{(\ell)}|<C\rho^{j}$.
\end{proof}
\begin{cor}
\label{cor:to:bary}
Under the assumptions of Proposition~\ref{prop:KSV:joining} there
exist $ \rho<1$, $C'>0$ so that $d_{KR}(\nu_k^{(\ell)},\frac 1
d\sum_{\ell=1}^d \nu_{b}^{(\ell)})\leq C'\sum_{j=b}^k\epsilon_j
+C'\rho^{k-b}$ whenever $k\geq b$ and $\ell\in \{1,\dots,d\}$.
\end{cor}
\bold{Remark.} Corollary~\ref{cor:to:bary} establishes all the
conclusions of  Proposition~\ref{prop:KSV:joining}
except the ergodicity of $\mu$.
\begin{proof}[Proof of Corollary~\ref{cor:to:bary}]
First notice that by \ref{cond:track}  we have that 
\begin{equation}\label{eq:bary1}
d_{KR}(\nu_j^{(\ell)}|_{A_j},\nu_{j-1}^{(\ell-1)}|_{A_j})<\epsilon_j\text{ and }d_{KR}(\nu_j^{(\ell)}|_{B_j},\nu_{j-1}^{(\ell)}|_{B_j})<\epsilon_j.
\end{equation}
We now claim that for all $\ell$, 
\begin{equation}
\label{eq:KR est2}
d_{KR}\left(\frac 1 {\lambda(A_j)} \nu_{j-1}^{(\ell)}|_{A_j},\nu_{j-1}^{(\ell)}\right)<\epsilon_{j-1}+2\epsilon_j+\frac {\epsilon_j} {c^2}
\end{equation}
Indeed, for $f$ 1-Lipschitz with $\|f\|_{\sup}\leq 1$ we have
\begin{multline*}
\frac{1}{\lambda(J_j)r_j}\int_{A_j}f \,
d\nu_{j-1}^{(\ell)}=\frac{1}{\lambda(J_j)r_j}\int_{\bigcup_{i=1}^{r_j}S^iJ_j\setminus U}f \,
d\nu_{j-1}^{(\ell)}=\frac{1}{\lambda(J_j)r_j}\sum_{i=1}^{r_j}\int_{J_j}f \circ S^i(x)
\chi_{U^c}(S^ix)\, d\nu_{j-1}^{(\ell)}\\
=\frac{1}{\lambda(J_j)r_j}\sum_{i=1}^{r_j}\int_{J_j}f \circ S^i(x)
d\nu_{j-1}^{(\ell)}-\frac{1}{\lambda(J_j)r_j}\sum_{i=1}^{r_j}\int_{J_j}f \circ S^i(x)
\chi_{U}(S^ix)\, d\nu_{j-1}^{(\ell)}.
\end{multline*}
By \ref{cond:kr close}
\begin{displaymath}
\left|\frac{1}{\lambda(J_j)r_j}\sum_{i=1}^{r_j}\int_{J_j}f \circ S^i(x)
d\nu_{j-1}^{(\ell)} - \int f \,
d\nu_{j-1}^{(\ell)}\right| \le \epsilon_{j-1},
\end{displaymath}
and by \ref{cond:U small} (i.e.\ the size estimate on $U_j$),
\begin{displaymath}
\left|\frac{1}{\lambda(J_j)r_j}\sum_{i=1}^{r_j}\int_{J_j}f \circ S^i(x)
\chi_{U}(S^ix)\, d\nu_{j-1}^{(\ell)}\right| \le \|f\|_{\sup} \lambda(U_j
\cap \bigcup_{i=1}^{r_j}S^i J_j ) \le 2 \epsilon_j.
\end{displaymath}
Then (\ref{eq:KR est2}) follows because 
$$\left|\frac{1}{r_j\lambda(J_j)}-\frac 1 {\lambda(A_j)}\right|\leq \left|\frac 1 {r_j\lambda(J_j)}-\frac{1}{r_j \lambda(J_j)-\lambda(U_j)}\right|\leq \frac{\epsilon_j}{c^2}.$$

Similarly, by partitioning $B_j$ into $D_{\frac{r_j}2},...$ where 
$$D_\ell=\{x\in S^{r_j}J_j: \min\{i>0:S^ix\in J_j\}=\ell\},$$ 
we get
\begin{equation}
\label{eq:KR:est:2prime}
d_{KR}\left(\frac 1 {\lambda(B_j)} \nu_{j-1}^{(\ell)}|_{B_j},\nu_{j-1}^{(\ell)}\right)<\epsilon_{j-1}+2\epsilon_j+\frac {\epsilon_j} {c^2}.
\end{equation}

 So for any 1-Lipschitz function, $f$, with $\|f\|_{\sup}\leq 1$, we
 claim that we may apply Lemma~\ref{lemma:to:bary} to 
$\gamma_i^{(\ell)}=\int fd\nu_{i}^{(\ell)}$ with $c=c$,
$\delta_{j-1}=\epsilon_{j-1}+4\epsilon_j+\frac {\epsilon_j} {c^2}$,
$a_j=\lambda(A_j)$ and $b_j=\lambda(B_j)$. To verify
(\ref{eq:gammas:assumption}), note that
$$\left|\int f
d\nu_i^{(\ell)}-\int_{A_i}fd\nu_i^{(\ell)}-\int_{B_i}fd\nu_i^{(\ell)}\right|\leq
\|f\|_{\sup}\lambda(U_i)<\epsilon_i$$ and so by (\ref{eq:bary1})
$$\left|\int fd \nu_i^{(\ell)}-\int_A f\, d\nu_{i-1}^{(\ell-1)}-\int_B f
d\nu_{i-1}^{(\ell)}\right|\leq 2\epsilon_i.$$
Then, by (\ref{eq:KR est2}) and (\ref{eq:KR:est:2prime}),
$$\left|\int fd \nu_i^{(\ell)}-\Big(\lambda(A_i)\int
fd\nu_{i-1}^{(\ell-1)}+\lambda(B_i)\int f
d\nu_{i-1}^{(\ell)}\Big)\right|\leq \epsilon_{j-1}+4\epsilon_j+\frac
{\epsilon_j} {c^2}.$$
This completes the verification of (\ref{eq:gammas:assumption}), and,
in view of Lemma~\ref{lemma:to:bary}, the proof of
Corollary~\ref{cor:to:bary}. 
\end{proof}

To complete the proof of Proposition~\ref{prop:KSV:joining}, we need
to prove that $\mu$ is ergodic. We start with the following:

\begin{lemma}
\label{lemma:sufficient}
It suffices to show that for any
$\epsilon>0$ and $M\in \natls$ there exists $c > 0$ and $G\subset Y \times Y$ with $\mu(G)>c$ and so that for $(x,y)\in G$ there exists $L>M$ with 
$d_{KR}(\frac 1 L \sum_{i=1}^L\hat{\delta}_{(S\times S)^i(x,y)},\mu)<\epsilon$. 
\end{lemma}

To prove Lemma~\ref{lemma:sufficient} we use the following consequence of the ergodic decomposition.
\begin{lemma}
\label{lemma:ergodic:full:measure}
Let $\tilde{T}:\tilde{Y} \to \tilde{Y}$ be a measurable map of a $\sigma$-compact metric space and $\tilde{\mu}$ be a invariant measure. For $\tilde{\mu}$ almost every $z \in \tilde{Y}$ we have that
 $\frac 1 N\sum_{i=0}^{N-1} \delta_{\tilde{T}z} $ converges to an ergodic measure in the weak-* topology. (The measure is allowed to depend on the point.) 
\end{lemma}
\begin{proof} $\tilde{\mu}$ has an ergodic decomposition $\tilde{\mu}=\int_{\tilde{Y}}\tilde{\mu}_y d\tilde{\mu}$ where $\tilde{\mu}_y$ 
is an ergodic probability measure with $\tilde{\mu}_y(\{z: \tilde{\mu}_z=\tilde{\mu}_{y}\})=1$ for $\tilde{\mu}$-almost every $y$. For each $y$, let 
$$Z_y=\{z:\underset{N \to \infty}{\lim}\, \frac 1 N \sum_{i=0}^{N-1}f(T^ix)=\int f d\tilde{\mu}_y \text{ for every }f\in C_c(Y)\}.$$
Because there is a countable $\|\cdot\|_{\sup}$-dense subset of $C_c(Y)$, by the Birkhoff Ergodic Theorem, we have that $\tilde{\mu}_y(Z_y)=1$ for all $y$. 
$\cup Z_y$ has full $\tilde{\mu}$ measure and satisfies the conclusion of the lemma. 
\end{proof}

\begin{proof}[Proof of Lemma~\ref{lemma:sufficient}]
By our assumptions, a positive $\mu$ measure set of $(x,y)$ have that $\mu$ is a weak-* limit point (in particular the set lim sup of the $G$ for a choice of $\epsilon$ going to 0). Throwing out a set of $\mu$ measure zero where the limit may not exist, Lemma~\ref{lemma:ergodic:full:measure} implies this is the unique weak-* limit point and it is ergodic.
\end{proof}

We now identify a set of full measure for $\mu$. As a preliminary, by
the assumptions \ref{cond:small error} and \ref{cond:track} of
Proposition~\ref{prop:KSV:joining} we have that if $x \notin
\bigcap_{n=1}^{\infty}\bigcup_{k=n}^\infty U_k$ (this is a full
measure condition) then there exists $p_1(x),...,p_{d}(x)$ so that
$$\underset{i \to \infty}{\lim} \{(\nu_i^{(\ell)})_x\}_{\ell=1}^d=\{\hat{\delta}_{p_1(x)},\dots,\hat{\delta}_{p_{d}(x)}\}.$$

\begin{lemma} $\mu\Big(\big\{\big(x,p_1(x)),...,(x,p_{d}(x)\big)\big\}_{x\notin
  \cap_{n=1}^{\infty}\cup_{k=n}^\infty U_k}\Big)=1.$ 
\end{lemma}
\begin{proof}
It is straightforward to see that for any $f \in C(Y\times Y)$ we have 
$$\underset{i \to \infty}{\lim}\int_{Y\times Y} fd(\sum_{\ell=1}^d \frac 1 d \nu_i^{(\ell)})=\int_Y\frac{1}{d}\sum_{\ell=1}^df(x,p_\ell(x))d\lambda.$$
  By Corollary \ref{cor:to:bary}, the left hand side is $\int f\, d\mu
  = \underset{i \to \infty}{\lim}\int_{Y\times Y} fd \nu_i^{(\ell)}$ for every $\ell$, establishing the lemma. 
\end{proof}
\begin{proof}[Proof of Proposition~\ref{prop:KSV:joining}]
Let $G_k$ be the set of all $x \in Y$ so that 
\begin{enumerate}
\item $S^ix \notin \bigcup_{j=k+1}^\infty J_j \cup S^{r_j}J_j$ for all $0\leq i \leq \frac{r_{k+1}}9$
\item $|\{0\leq i\leq r_{k+1}: S^jx \in \bigcup_{j=k}^\infty
U_j\}|<4\sum_{j=k}^\infty \frac {\epsilon_k}9 r_{k+1}$. 
\item $x \notin \bigcup_{j=k}^\infty U_j$.
\end{enumerate}

\begin{claim}For all large enough $k$ we have that $\lambda(G_k)\geq \frac 1 2 $. 
\end{claim}
This is a straightforward measure estimate using Assumptions \ref{cond:U small}, \ref{cond:scale grow} and \ref{cond:small error}.
\medskip

Suppose $x \in G_k$ and $y \in supp(\mu_x)$.  The next claim shows that  there exists $\ell_j$ so that
$\lim (\nu_j^{(\ell_j)})_x$ is the point mass at $y$.
\begin{claim}
\label{claim:1}
There exists $\ell'$ so that $d_{KR}(\hat{\delta}_y,(\nu_{k}^{(\ell')})_x)<3\sum_{j=k+1}^\infty \epsilon_j$. Also 
$$d_{KR}\left(\frac 9 {r_{k+1}}\sum_{i=1}^{\frac{r_{k+1}}9}\hat{\delta}_{(S\times S)^i(x,y)}, \nu_k^{(\ell')}\right)<C''\sum_{j=k}^\infty \epsilon_j.$$ 
\end{claim}
\begin{proof}[Proof of Claim~\ref{claim:1}]
We first state the following straightforward consequence of the condition \ref{cond:track} of Proposition~\ref{prop:KSV:joining} (by considering if $x \in A_k$ or $x\in B_k$):
\begin{lemma}Let $a \in \{1,..,d\}$. If $S^jx\notin J_k \cup S^{r_k}J_k$ for $0\leq j\leq L$ then there exists $\ell$ (it is either $a$ or $a+1$) so that 
$d_{KR}((\nu_k^{(\ell)})_{S^jx},(\nu_{k-1}^{(a)})_{S^jx})<\epsilon_k$ for any $0\leq j\leq L$ with $S^jx \notin U_k$. 
\end{lemma}
By iterating we obtain:
\begin{cor}\label{cor:dist move} For all $j>k, \, \ell \in \{1,...,d\}$ and $x\in G_k$ %we have that if $S^ix \notin \cup_{s=k+1}^j J_s \cup S^{r_s}J_s$ then
 there exists $\ell'$  so that
$d((\nu_k^{(\ell)})_{S^ix},(\nu_j^{(\ell')})_{S^ix})<2\sum_{s=k+1}^j
\epsilon_s$ for any $0\leq i\leq \frac {r_{k+1}}9$ with $S^i x \notin
\bigcup_{s=k+1}^j U_s$.  
\end{cor}
Note that if $L\geq \frac{r_{k+1}}9$ by condition \ref{cond:kr close} of the proposition we obtain 
\begin{equation}
\label{eq:close measures} d_{KR}(\frac 1 L\sum_{j=1}^{L}(\nu_{k}^{(a)})_{S^jx},\nu_{k}^{(a)})<\epsilon_{k}.%+\frac 1 L|\{1\leq j\leq L:S^jx\in U_{k+1}\}|.
\end{equation}

By Corollary \ref{cor:dist move} there exists $\ell$ so that if $S^ix
\notin \bigcup_{\ell=k+1}^\infty U_\ell$ then for some $\ell$ we have
$d_{KR}(\delta_{(S^i \cross S^i)(x,y)},(\nu_k^{(\ell)})_{S^i x})\leq \sum_{j=k+1}^\infty \epsilon_j$ (for $0\leq i\leq \frac{r_{k+1}}9$). 
With \eqref{eq:close measures} this gives 
$$d_{KR}(\sum_{i=1}^{\frac{r_{k+1}}9}\delta_{(S\times S)^i(x,y)},
\nu_{k}^{(a)})< \epsilon_k+2 \sum_{j=k+1}^\infty
\epsilon_j+4\sum_{j=k+1}^\infty \epsilon_j .$$
%The claim follows from this corollary, Equation (\ref{eq:close measures}) and the fact that $|\{0\leq i\leq r_{k+1}: S^jx \in \cup_{j=k}^\infty U_j\}|<\frac {\epsilon}9 r_{k+1}$. 
\end{proof}

This completes the proof by verifying Lemma~\ref{lemma:sufficient} since for all $\epsilon>0$ there exists $k_0$ so that for all $k\geq k_0$ and $\ell \in \{1,...,d\}$ we have 
$d_{KR}(\mu,\nu_k^{(\ell)})<\epsilon$ (by Corollary \ref{cor:to:bary}).
\end{proof}

\section{Proof of Theorem \ref{theorem:Veech}}\label{sec:theorem1}
In this section, we will verify the conditions of
Proposition~\ref{prop:KSV:joining}. 

Before beginning the proof we set up a geometric context connected to our situation. 
A 3-IET with lengths $\ell_1,\ell_2$ and $\ell_3$ is a rescaling of the Poincar\'e first return map of 
rotation by $\frac{\ell_2+\ell_3}{\ell_1+2\ell_2+\ell_3}$ to the interval $[0,\frac {\ell_1+\ell_2+\ell_3}{\ell_1+2\ell_2+\ell_3})\subset [0,1)$ \cite[Section 8]{KS}. 
If $\omega_{sq}$ denotes the  area one square torus oriented horizontally and
vertically, observe that rotation by $\alpha$ corresponds to the first return map of the vertical flow on $\begin{pmatrix} 1&-\alpha\\0&1\end{pmatrix}\omega_{sq}$ to a horizontal side, which is also the time one map of that flow.% $F^1_{\begin{pmatrix} 1&-\alpha\\0&1\end{pmatrix}\omega_{sq}}.$

To set up the geometric context, 
let $\mathcal{M}_{1,2}$ denote the moduli space of area 1 tori with
two marked points. Note that $\cM_{1,2}$ is isomorphic to
$(SL(2,\reals) \ltimes \reals^2)/(SL(2,\zed) \ltimes \zed^2)$. 
%We denote elements of this group by $(A,\begin{pmatrix}s\\t\end{pmatrix})$. 
For $\omega
\in \mathcal{M}_{1,2}$ let $F^t_\omega$ denote the vertical flow on
$\omega$, which corresponds to left multiplication by the element
$\begin{pmatrix} 1  & 0 \\ 0 & 1 \end{pmatrix} \ltimes \begin{pmatrix}
0 \\ t \end{pmatrix}$. Let $\hat{\omega}\in \mathcal{M}_{1,2}$ be the
square torus with two marked points distance $\frac 1 2 $ apart on the
same horizontal line segment.   Let $\mathcal{S} \subset
\mathcal{M}_{1,2}$ be the set of surfaces $\omega$ so that $F^1_{\omega}p$ is on the same horizontal as $\omega$ and its distance along this horizontal is at most $\frac 1 2 $.  That is, if $p$ is one marked point the other marked point is at $\begin{pmatrix} 1  & 0 \\ 0 & 1 \end{pmatrix} \ltimes\begin{pmatrix}s\\0\end{pmatrix}$ where $s\leq \frac 1 2 $. 
%{\sc Is there a better way to say this?}
% Let $\mathcal{S}\subset \mathcal{M}_{1,2}$ be the set of translation surfaces so that $F_{\omega}^1(p)$ is translated by $p$ along the horizontal by at most $\frac 1 4$. 

Let $T$ be a 3-IET. It arises as the first return map of a rotation
$R_\alpha$ to an interval $K$. Let 
\begin{displaymath}
\psi_M(x) = \sum_{\ell=0}^{M -1} \chi_K(R_\alpha^\ell x).
\end{displaymath}
Then, for any $x \in K$ so that $R_\alpha^M x\in K$,
\begin{equation}
\label{eq:time:change:general}
T^{\psi_M(x)} x = R_\alpha^M x.
\end{equation}

  Let $\omega_T\in \mathcal{M}_{1,2}$ be the torus defined by taking the torus 
$\begin{pmatrix} 1&-\alpha\\0&1\end{pmatrix}\omega_{sq}$ and 
% where
%$\omega_{sq}$ is the  area one square torus oriented horizontally and
%vertically, and
 marking two points on the bottom horizontal line that
are $|K|$ apart. 
 When it is convenient, in what follows we will
consider $K$ as being embedded in $\omega_T$ and $g_t K$ as being
embedded in $g_t \omega_T$ where
$g_t=\begin{pmatrix} e^t&0\\0&e^{-t}\end{pmatrix}$. Here we are
identifying $g_t$ with the matrix $\begin{pmatrix}
e^t&0\\0&e^{-t}\end{pmatrix} \ltimes \begin{pmatrix} 0 \\
0 \end{pmatrix}$ and think of $g_t$ as acting on $\cM_{1,2} \isom
SL(2,\reals) \ltimes \reals^2/SL(2,\zed) \ltimes \zed^2$ by left
multiplication. 
Thus, for any $M\in \natls$ we can identify $\psi_M(x)$ as the
intersection number between $K$ and a vertical line of length $M$ on $\omega_T$
starting at a $x$, see Figure~\ref{fig:original}. Using this as a
definition, we can make sense of $\psi_M$ for all $M \in
\reals^+$. 

If we embed $K$ in $\omega_T$, then for $x \in K$ and $M \in \natls$, we have
\begin{equation}
\label{eq:iet:vs:vert}
T^{\phi_M(x)} = F^{M}(x) \quad\text{{if $F^M(x) \in K$}}
\end{equation}
where $\phi_M(x)$ is the number of intersections between a vertical
line of length $M$ starting at $x$ and $K$.

\begin{figure}[ht]
    \centering
    \includegraphics[width=0.3\textwidth]{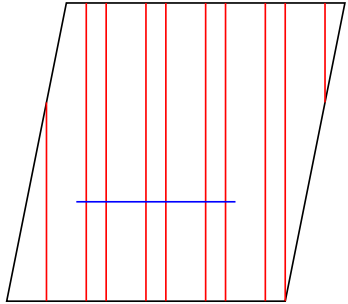}
    \caption{The torus $\omega_T$. A vertical segment of length $M$
      (colored in red)
      intersects a horizontal slit (colored in blue) of length $|K|$.}
    \label{fig:original}
\end{figure} 

\begin{lemma}
\label{lemma:hatw:limit:point}
For almost every $T$ we have that $\hat{\omega}$ is a
limit point of $\{g_t\omega_T\}_{t \geq 0}.$ 
\end{lemma}

\begin{proof}
Let $U^+$ denote the subgroup $\begin{pmatrix} 1 & \ast \\ 0 &
1 \end{pmatrix} \ltimes \begin{pmatrix} \ast \\ 0 \end{pmatrix}$ of
$SL(2,\reals) \ltimes \reals^2$. Then, $U^+$ is the expanding
horospherical subgroup with respect to the action of $g_t$, or in
other words, the orbits of $U^+$ are the unstable manifolds for the
flow $g_t$. 

By construction, map $T \to \omega_T$ projects to a positive measure
subset $\cD$ of a single $U^+$ orbit on $\cM_{1,2}$. 
% \footnote{$\{\begin{pmatrix}
%   1&-\alpha\\0&1\end{pmatrix}\omega_{sq}\}_{\alpha\in [0,1)}$ is an
%   unstable leaf in $\mathcal{M}_1$. Because $g_t$ contracts vertical
%   distances, if $\omega,\omega'\in \mathcal{M}_{1,2}$ project to the
%   same central-stable leaf in $\mathcal{M}_1$ and their marked points
%   are on the same vertical line then they are the same central-stable
%   leaf in $\mathcal{M}_{1,2}$.} 
Moreover, the pushforward of the Lebesgue measure on the space of
3-IET's to $\cD$ 
is absolutely continuous with respect to the pushforward of the Haar
measure on $U^+$ to $\cD$. The lemma then follows from the ergodicity of $g_t$. 
\end{proof}

\begin{corollary}
\label{cor:hatw:cap:S}
For every $\delta > 0$ and almost all $T$, there exists arbitrarily
large $t>0$ with $g_t \omega_T \in B(\hat{\omega},\delta) \cap \cS$.  
\end{corollary}

\bold{Proof.}  Since $\hat{\omega}$ is square, for $p
\in \hat{\omega}$, $F^1_{\hat\omega}p=p$. Therefore, for $\omega' \in
B(\hat{\omega},\delta')$ and $p \in \omega'$, 
$F^1_{\omega'} p$ is within $c_1(\delta')$ of $p$, where $c_1(\delta') \to 0$ as
$\delta' \to 0$. Write $F^1_{\omega'} p - p = (v_1, v_2)$, and note that for
$\delta'> 0$ sufficiently small and for small $s \in \reals$, for $p \in g_s
\omega'$, we have 
\begin{displaymath}
F^{1}_{g_s\omega'} p - p = (e^{-s} v_1, 1- e^{s} + e^{s} v_2). 
\end{displaymath}
Therefore, given $\omega' \in B(\hat{\omega}, \delta)$, 
we can choose $s \in \reals$, with $|s| < c_2(\delta')$ where
$c_2(\delta') \to 0$ as $\delta \to 0$, such that $1-e^s + e^s v_2 =
0$, i.e.\ $g_s \omega' \in
\cS$. We have $g_s \omega' \in B(\hat{\omega}, c_3(\delta'))$ with
$c_3(\delta') \to 0$ as $\delta' \to 0$.  

Suppose $T$ is such that $\hat{\omega}$ is a limit point
of $\{g_t \omega_T \}_{t \ge 0}$. Choose $\delta' > 0$ such that
$c_3(\delta') < \delta$ and choose $t'$ such that $g_{t'}
\omega_T \in B(\hat{\omega}, \delta')$ and then let $t = t'+s$ where
$s$ is as in the previous paragraph. Then $g_t \omega_T \in
B(\hat{\omega},\delta) \cap \cS$ as required. 
\qed\medskip

We now apply $g_t$ to Figure~\ref{fig:original}, with $t = \log M$.  
Note that $\psi_M(x)$ is also the intersection number between a
vertical segment $\gamma_1$ of length $1$ and a horizonal slit
$\gamma_2$ of length $M |K|$
(see Figure~\ref{fig:renorm1}). From now on, we assume that $g_t
\omega_T \in B(\hat{\omega}, \delta) \cap \cS$ for some $\delta \ll
1$.

\begin{figure}[ht]
    \centering
    \includegraphics[width=0.3\textwidth]{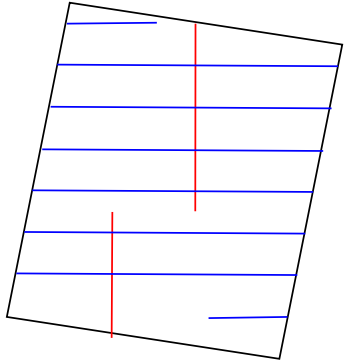}
    \caption{The torus $g_{\log M} \omega_T$: A vertical segment
      $\gamma_1$ of length $1$ (drawn in red) intersects a horizontal
      slit $\gamma_2$ of length $M |K|$ (drawn in blue). If we also
      assume that $g_{\log M} \omega_T \in B(\hat{\omega}, \delta)
      \cap \cS$ then the torus $g_{\log M} \omega_T$ is nearly square,
      and the two endpoints of $\gamma_1$ are on the same horizontal
      line segment, of length at most $1/2+O(\delta)$.}
    \label{fig:renorm1}
\end{figure}

%\begin{prop} If $\hat{\omega}$ is a limit point of $\{g_t\omega_T\}_{t\geq0}$ then $T$ satisfies the the conclusions of the main theorem.
%\end{prop}

\begin{figure}[ht]
    \centering
    \includegraphics[width=0.3\textwidth]{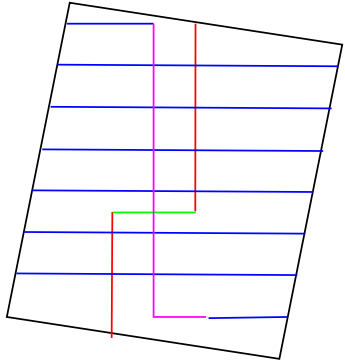}
    \caption{Closing the curves. We complete the vertical segment 
     $\gamma_1$ to a closed
      curve $\hat{\gamma_1}$ by adding a horizontal segment $\zeta_1$
      (drawn in green). Note that since $g_t \omega_T \in
      B(\hat{\omega}, \delta) \cap \cS$, the length of $\zeta_1$ is at
      most $1/2+O(\delta)$. Similarly, we close up the horizontal slit
      $\gamma_2$ to obtain a closed curve $\hat{\gamma}_2$ 
      by adding in a horizontal segment $\zeta_2$ and a
      vertical segment $\zeta_2'$ (drawn in purple).} 
    \label{fig:renorm2}
\end{figure}

The following lemma references Figure~\ref{fig:renorm2}.
\begin{lemma} \label{lem:crossing} There exists $m$ so that if the
green segment does not cross the purple segment then the number of
times a trajectory of length 1 crosses $g_t K$ (the blue lines) is
either $m$ or $m+1$. Moreover it is $m+1$ if it does not cross the
(horizontal) purple segment and $m$ if it does.

In other words, for the set of points $x$ whose green segment does
not cross the purple segment, ${\phi_M(x)}$ is $m$ if its red segment
crosses the purple segment (where $\phi_M$ is as in
\eqref{eq:iet:vs:vert}) and $\phi_M(x) = m+1$ if it does not.  
\end{lemma}
Note
that because Figure \ref{fig:renorm2} is of $g_{\log(M)}\omega_T$,
vertical trajectories of length 1 in Figure \ref{fig:renorm2}
correspond to vertical trajectories of length $M$ on $\omega_T$.
\begin{proof}
Indeed, the family of curves we define are all homotopic and so their intersection with $\hat{\gamma}_2$ is all the same. So for such curves, if the green and purple segments have intersection number zero then the intersection of the red segment and blue segment depends only on the intersection of the purple segment and the red segment, which by construction is either 0 or 1.
\end{proof}
\begin{lemma}
\label{lem:translation}
For all $\epsilon>0$ there exists
$\delta>0$ so that if $\omega \in B(\hat{\omega},\delta) \cap
\mathcal{S}$ and the flow $F^s_\omega$ is minimal then there exists
$\rho<\epsilon$ and $L \in \natls$ so that for any interval $J$ with $|J|=\rho$ we have
\begin{itemize}
\item $\lambda\left(\bigcup_{s\in [0,L)}F^sJ\right)>1-\epsilon$
\item For all $0\leq s<\ell<L$ we have $F^sJ\cap F^\ell J=\emptyset$
\item $F^1J$ is horizontally adjacent to $J$. 
\end{itemize}
\end{lemma}
\begin{proof}
Suppose $p$ is a point in $\omega$, 
and $\omega \in \mathcal{S}$. Then, $F^1p$ is horizontally adjacent to
$p$. For all $\epsilon>0$ there exists $\delta>0$ so that if $\omega$
is in $\mathcal{S} \cap B(\hat{\omega},\delta)$ then $F^1p$ is translated by less than $\frac \epsilon 9$. 
Since the vertical flow on $\omega$ is minimal,  $F^1p \ne
p$. Therefore, $F^1p$ is translated horizontally by some amount $\rho>0$. Let $J$ be a horizontal interval of length $\rho$. We choose $L=\min\{s>0:F^sJ \cap J \neq \emptyset\}$. We have that $\lambda\left(\bigcup_{s\in [0,L)}F^sJ\right)>1-\rho$.  Indeed, 
$\bigcup_{s\in [0,L+1)}F^sJ=\omega$. 
\end{proof}

\begin{prop}\label{prop:two meas}For any $a,b \in \mathbb{Z}$ and $\frac c 5>\epsilon>0$ there exists
 $\delta>0$, 
 $t_0>0$ so that if $g_t\omega_T\in B(\hat{\omega},\delta) \cap \mathcal{S}$, $\lambda(K)>c$ and $t>t_0$ then there exist 
 \begin{itemize} 
 \item $n \in \mathbb{Z}$, $r \in \natls, L \in \mathbb{N}$  
 \item an interval $J \subset K$, and a measurable set $B \subset K$
 \end{itemize} so that the minimal return time (under $T$) to $J$ is
 at least $\frac 3 2 r$, and 
 for $A=\bigcup_{i=0}^rT^iJ $ we have
 $\lambda(A)>\frac 1 2 \lambda(K)-\epsilon,\lambda(B)>\frac 1 2
 \lambda(K)-\epsilon$ and the sets $A$ and $B$ satisfy
$$d(T^nx,T^ax)<\epsilon \text{ for all }x\in A$$ and 
$$d(T^nx,T^bx)<\epsilon \text{ for all }x \in B.$$ 
Moreover, $T^iJ \cap T^jJ=\emptyset$ for all $0\leq i<j\leq \frac 3 2 r$. 
Lastly, if $\nu^{(a)}$ is the joining supported on $\{(x,T^ax)\}$ then for all $x \in A$ we have 
$$d_{KR}\left(\frac 1 L\sum_{i=0}^{L-1}\delta_{(T^ix,T^{i+n}x)},\nu^{(a)}\right)<2\epsilon $$
and if $\nu^{(b)}$ is the joining supported on $\{(x,T^bx)\}$  then for all $x \in B$ we have 
$$d_{KR}\left(\frac 1 L\sum_{i=0}^{L-1}\delta_{(T^ix,T^{i+n}x)},\nu^{(b)}\right)<2\epsilon. $$
\end{prop}
\begin{rem}\label{rem:simple} Specializing to the case where $a=0$ and $b=k$, we see that $\frac 1 2 (Id+T^k)$ is in the weak closure of the powers of $T$. 
Veech showed that almost every 3-IET has simple spectrum \cite[Theorem 1.3]{veech metric}. Combining these two facts with Ryzhikov's \cite[Theorem 6.1 (3) and (4)]{ryzh3} 
we have that the spectrum of $T^n$ and $\underbrace{T\times...\times T}_{n \text{ times}}$ are simple for all $n>0$. 
\end{rem}

\begin{proof}
In view of Lemma~\ref{lem:translation}, we can
choose $\delta $ so small that for any $\omega\in B(\hat{\omega},\delta)$,
\begin{enumerate}[label=(\roman*)]
\item The horizontal purple line has length between $\frac 1 2-\frac \epsilon 4$ and $\frac 1 2 +\frac \epsilon 4$.
\item\label{i:shift} $F^1_{\omega}x=\begin{pmatrix} 1  & 0 \\ 0 &
1 \end{pmatrix} \ltimes\begin{pmatrix}\rho\\0\end{pmatrix}x$  where $0
< \rho\leq\frac{\epsilon}{10(|a|+|b|)}$. 
\end{enumerate}
Because $T \times T$ is uniquely ergodic on the support of $\nu^{(a)}$ and $\nu^{(b)}$ there exists $L_0$ so that if 
$d(p_i,T^{a+i}y)<\epsilon$ for all $0\leq i\leq L$ then 
\begin{equation}\label{eq:KR close}
d_{KR}\left(\frac 1 L
\sum_{i=0}^{L-1}\delta_{(T^iy,p_i)},\nu^{(a)}\right)<2\epsilon, \qquad d_{KR}\left(\frac 1 L
\sum_{i=0}^{L-1}\delta_{(T^iy,p_i)},\nu^{(b)}\right)<2\epsilon
\end{equation}
for all $L\geq L_0$ and $y \in K$. Indeed, $T\times T$ is uniquely ergodic on $\text{supp}(\nu^{(a)})$ and $\text{supp}(\nu^{(b)})$ and uniquely ergodic systems have uniform convergence of Birkhoff averages of continuous functions 
 (see for example \cite[Proposition 4.7.1]{BS}). We choose $t_0$ so large that any vertical trajectory of length $e^{t_0}$ on $\omega_T$ crosses $K$ at least $L_0$ times. We further assume $L_0>\max\{|a|,|b|\}$. 

We now set about defining $J$ and $A$. Let $V$ be the horizontal
purple line segment.  Let $\rho$ be as in the previous lemma for
$g_t\omega_T$. For any horizontal interval $I$ on $g_t\omega_T$ of length $\rho$ we have one of the following mutually exclusive possibilities:
\begin{enumerate}[label=(\alph*)]
\item\label{cond:no purple}  $\bigcup_{s\in [0,1)}F^s(I) \cap  V=\emptyset$
\item \label{cond:purple} There exists $s \in [0,1)$ so that $F^s(I) \subset V$. 
\item  $\bigcup_{s\in [0,1)}F^s(I) \cap (\partial V)\neq\emptyset$
\end{enumerate}
Note that by Lemma \ref{lem:crossing}  there exists $m$ so that if $I \subset K$ so that if \ref{cond:purple} holds then $\phi_M(x)=m$ and similarly   if $I \subset K$ so that \ref{cond:no purple} holds then $\phi_M(x)=m+1$.

Let $\hat{A}$ be the set of points in $g_t\omega_T$ which belong to some
horizontal interval of length $\rho$ satisfying \ref{cond:purple}. Let 
\begin{equation}
\label{eq:in hatA}
\tilde{A}= \bigcap\limits_{s\in [-2-|a-b|,2+|a-b|]} F^s_{g_t\omega_T}\hat{A}.
\end{equation} 

%Let 
%$\hat{A}$ be the set of points $x$ on $g_tK$ so that $F_{g_t\omega_T}^\ell F^{s}_{g_t\omega_T}x$ crosses from the horizontal purple strip for some $0\leq \ell\leq 1$ and $-1\leq \ell \leq 0$ and for all $-2-|a-b|\leq s\leq 2+ |a-b|$. 
 %Let $\tilde{A}=\cap_{s\in[-1,1]}F_{g_t\omega_T}^s\hat{A}$.
  Let $\rho>0$ be given by Lemma \ref{lem:translation} and $I$ be an
  interval of length $\rho$ in $\tilde{A}  \cap g_t K$ so that $F^{-1}I\not\subset
  \tilde{A}$. Now $F^1I$ is horizontally adjacent to $I$, and so
  $F^jI$ is horizontally $j \rho$ over from $I$. So by our assumption on the length of $I$, we have 
\begin{equation}
\label{eq:FjI:subset:tildeA}
F^jI \subset\tilde{A}, \quad\text{ for all }\quad 0\leq j \leq \frac{|V|}\rho-2(2+|a+b|)-3\equiv\hat{p}.
\end{equation}
(Note that by \ref{i:shift} and the fact that $|V|>\frac 1 2 -\epsilon
$ we have $\hat{p} \ge 1$.)

We now use what we have done for the flow on $g_t\omega_T$ to establish some of our claims about the IET, $T$. 
Let $r$ be the cardinality of the  set of intervals of length $\rho$
in $\bigcup_{s\in [0,\hat{p})}F^s I \cap K$.  Note that because in
our set $\hat{A}$
a vertical trajectory of length 1 crosses $g_tK \subset
g_t\omega_T$ exactly $m$ times, $r=m\hat{p}$.

Let $A'=g_{-t}\hat{A}\cap K \subset \omega_T \cap K$, which we can
consider as a subset of the domain of $T$ as well (because it is contained in $K$).  Note that we have 
\begin{equation}
\label{eq:phi:et:is:m} \phi_{e^t}(x) = m \qquad \text{for $x \in A'$.}
\end{equation}
We also have for all $x \in A'$, 
\begin{equation}
\label{eq:d:Fet}
d(F^{e^t} x, x) = e^{-t} \rho,
\end{equation}
because when we apply $g_{-t}$ to pull back our dynamics from
$g_t\omega_T$ back to $\omega_T$ we contract horizontal distances by
$e^{-t}$. It follows from (\ref{eq:iet:vs:vert}),
(\ref{eq:phi:et:is:m}) and (\ref{eq:d:Fet}) that
\begin{equation}
\label{eq:shift}
d(T^m(x),x)=e^{-t}\rho \qquad\text{for all $x \in A'$.}
 \end{equation}

Let $J$ denote the interval corresponding to $I$ in the domain of our
IET, $T$. That is, we consider $J=g_{-t}I\subset K \subset \omega_T$,
which since it is in $K$ we consider as an interval in the domain of
$T$.
Let $A=\bigcup_{i=0}^{r-1}T^i J$, which we can consider as a subset of
$K \subset \omega_T$. We now claim that 
\begin{equation}
\label{eq:A:subset:Aprime}
A \subset g_{-t} \tilde{A} \cap K. 
\end{equation}
Indeed, by (\ref{eq:FjI:subset:tildeA}), we have
\begin{equation}
\label{eq:Fs:et:J}
F^{s e^t} J \subset g_{-t}\tilde{A} \qquad\text{for all  $0 \le s \le \hat{p}$.}
\end{equation}
It follows in view of (\ref{eq:phi:et:is:m}), that for $x \in J$, 
\begin{equation}
\label{eq:phi:hatp:et}
\phi_{\hat{p}e^t}(x) = \sum_{k=0}^{\hat{p}-1} \phi_{e^t}(F^{k e^t} x) =
m \hat{p} = r. 
\end{equation}
By (\ref{eq:iet:vs:vert}), we have for $x \in J$ and $i \in \natls$,  
\begin{displaymath}
T^i x = F^s x, \quad\text{where $s$ is such that $\phi_s(x) = i$. }
\end{displaymath}
Since for a fixed $x \in J$, the map $s \to \phi_s(x)$ is monotone
increasing in $s$, for $0 < i < r$ we have in view of (\ref{eq:phi:hatp:et}),
\begin{displaymath}
T^i x = F^s x \quad\text{where $s < \hat{p}$.}
\end{displaymath}
This, together with (\ref{eq:Fs:et:J}) implies
(\ref{eq:A:subset:Aprime}). The same argument shows that 
\begin{equation}
\label{eq:49:extended}
T^\ell x \in A' \qquad\text{for $x \in A$ and $|\ell| \le m(|a-b|+1)$.}
\end{equation}

%Then, by (\ref{eq:FjI:subset:tildeA}), we have
%\begin{displaymath}
%F^{j e^t} J \subset A' \qquad\text{for all  $0 \le j \le \hat{p}$.}
%\end{displaymath}
%It follows in view of (\ref{eq:phi:et:is:m}), that for $x \in J$ and $0 \le j \le \hat{p}$,
%\begin{displaymath}
%\phi_{e^t}(F^{je^t}x)= m
%\end{displaymath}
%for $0\le j \le \hat{p}$.

% and $A'=(g_{-t}\tilde{A}\cap K) \subset \omega_T$. 
We now claim that for all $x \in A$ we have: 
 $$d(T^nx,T^ax)\leq d(T^ax,T^ax)+
\sum_{i=1}^{|a-b|}d(T^{im+a}x,T^{(i-1)m+a}x)\leq \epsilon e^{-t}\leq \epsilon.$$
Indeed, by (\ref{eq:49:extended}) and \eqref{eq:shift} we have $d(T^{jm+a}x,T^{(j-1)m+a}x)=\rho e^{-t}$ for all $|j|\leq |a-b|$, because $|a|<m$.  We obtain the second inequality by \ref{i:shift}.

We now show that for all $x \in A$ 
$$d_{KR}(\frac 1 m \sum_{i=0}^{m-1}\delta_{(T^ix,T^{i+n}x)},\nu^{(a)})<2\epsilon.$$ By construction, if $x \in A$ then $T^ix \in A'$  for all $-m\leq i\leq m$. 
So we have that $d(T^{i+n}x,T^{i+a}x)<\epsilon$ for all $|i|\leq |m|$.
 So by \eqref{eq:KR close} and the fact that $m\geq L_0$ we have our condition on $d_{KR}$. 
 
 We now show that $\lambda(A)>\frac 1 2 \lambda(K)-\epsilon$. 
 This follows from the fact that by \ref{i:shift} the measure of the set of $x \in g_t\omega_T$ so that $F^{\ell}x$ crosses the horizontal purple strip for $0\leq\ell \leq 1$ and $-1\leq \ell\leq 0$ and $F^s_{g_t\omega_T}x$ does not have this property for some 
 $-1\leq s\leq 1$ has measure at most $2\frac{\epsilon}{10(|a|+|b|)}$. By our condition on the length of the purple horizontal strip, the measure condition on $A$ is completed. 
 
 The fact that the return time of $T$ to $J$ is at most $\frac 32 r$ follows from the fact that the measure of $A^c$ is at most $\frac 1 2\lambda(K) +\epsilon$ and so the orbit of $J$ after leaving $A$ and before returning to $J$ has measure at least 
 $\frac 1 2 \lambda(K)-\epsilon-\epsilon>\frac 1 2 \lambda(A)$. So $J$ has at least $\frac 1 2 r$ images outside of $A$ before part of it returns.

We now similarly define $B \subset A^c$ with the desired properties. First let 
$$\hat{B}=\{x \in g_t\omega_T : \cup_{s\in [-3-|a-b|,3+|a-b|]}F^s_{g_t\omega_T}(x) \cap V=\emptyset\}.$$ Similarly to before let $\tilde{B}=\cap_{s\in [-1,1]}F^s_{g_t\omega_T}\hat{B}$ and $B=g_{-t}\tilde{B} \cap K \subset \omega_T$, considered as a subset of the domain of $T$. Now as above, by Lemma \ref{lem:crossing} if $x \in \tilde{B}$ then we have that a vertical trajectory of length 1 or -1 emanating from $x$ crosses $g_tK$ exactly $m+1$ times. 
Moreover, $F^s_{g_t\omega_T}x$ has this property for all $-|a-b|\leq s\leq |a-b|$. Since $n=b+(m+1)(a-b)$, for any $x \in B$ and $|i|\leq m$ we have $d(T^nT^ix,T^bT^ix)\leq \sum_{i=1}^{|a-b|}d(T^{i(m+1)}x,T^{(i-1)(m+1)}x)\leq \epsilon$. 
Thus, as above we have $d_{KR}(\frac 1 m \sum_{i=0}^{m-1}\delta_{(T^ix,T^{i+n}x)},\nu^{(b)})<2\epsilon$ for all $x \in B$.
The fact that $\lambda(B)>\lambda(K)-\epsilon$ is similar to the case of $\lambda(A)$ above. 

\end{proof}

Now given two number $a,b$ we may iteratively apply Proposition \ref{prop:two meas} to obtain the assumptions of Proposition \ref{prop:KSV:joining}.  %Indeed, by our construction $A$ is chosen to be images of an interval $J$, a segment of $K$ that that hits.
Indeed, we choose $\epsilon_i$ satisfying assumptions (c) and (e). We apply Proposition \ref{prop:two meas} to the pair of numbers $(a,b)$ and $\epsilon=\epsilon_1$ to obtain $m$, $A,B$ and $r$. %Further we obtain $r$ from Corollary \ref{cor:adapted}.
Denote $m$ by $a_1$. 
We apply Proposition \ref{prop:two meas} to the pair of numbers $(a,b)$ and $\epsilon=\epsilon_1$ to obtain $m$, $r'$, $A',B'$,  and denote $m$ by $b_1$. % and $r'$ from Corollary \ref{cor:adapted}. 
We repeat this procedure with $a_1$ and $b_1$ in the place of $a$ and $b$ and $\epsilon_2$ in place of $\epsilon$ and obtain $a_2,b_2$. We further request that the interval $J$ produced by Proposition \ref{prop:two meas} %Corollary \ref{cor:adapted} 
 have $\max\{r,r'\}\lambda(J)<\epsilon_2$. Iterating this we have the conditions of Proposition.

\begin{proof}[Proof of Theorem \ref{theorem:Veech}]
Let $\mu$ be an invariant measure for $T \times T$. By Corollary \ref{cor:WOT close} there exists $n_1,...,n_d$ so that $\nu_i$ is the joining supported on $\{(x,T^{n_i}x)\}$ and $d_{KR}(\mu,\frac 1 d\sum_{i=1}^d\nu_i)<\epsilon$. 
For each pair $n_i,n_{i+1}$ and $\frac{\epsilon}2$ we apply Proposition \ref{prop:two meas} % and Corollar \ref{cor:adapted}
 to obtain $\delta$, $t_0$. We further do this for the pair $n_d,\, n_{1}$. We choose $\delta$ to be the smallest of these and $t_0$ to be the largest. We obtain $t>t_0$ so that $g_t\omega_T \in B(\hat{\omega},\delta)$. We then obtain $m_i$, $r_i$ which we denote $n_i^{(1)}$ and $r_i^{(1)}$. We now repeat this $n_i^{(1)}$ in place of $n_i$, $\frac \epsilon {2^2}$ in place of $\frac{\epsilon}2$ and $\max\{r_i^{(1)}\}$. % in place of no condition in invoking Corollary \ref{cor:adapted}.
  In doing this we obtain $n_i^{(2)}$ and $r_i^{(2)}$. We repeat this recursively having our $k$th choice of $\epsilon$ be $\frac \epsilon {2^k}$.

  We are now left to prove that there is an ergodic self-joining that is neither $\lambda \times \lambda$ nor one-to-one on almost every fiber. Let $\nu_0^{(1)}$ be the self-joining carried on $\{(x,x)\}$ and $\nu_0^{(2)}$ be the self-joining carried on $\{(x,Tx)\}$.
   Let $\epsilon_i>0$ satisfy that 
   \begin{equation}\label{eq:keep away} d(x,Tx)>40C\sum_{i=1}^\infty \epsilon_i
   \end{equation} and 
   \begin{equation}\label{eq:not prod} d_{KR}(\lambda \times \lambda, \frac 1 2 (\nu_0^{(1)}+\nu_0^{(2)}))>4 C \sum_{i=1}^\infty \epsilon_i,
   \end{equation} 
   where $C$ is as in the conclusion of Proposition \ref{prop:KSV:joining}. 
   We apply Proposition \ref{prop:KSV:joining} for these $\epsilon_i$ as above to obtain $\nu_i^{(1)},\nu_i^{(2)}$ and their weak-* limit $\nu_\infty$, an ergodic measure which by \eqref{eq:not prod} is not $\lambda \times \lambda$. 
   The following lemma show $\nu_\infty$ can not be one-to-one on almost every fiber. 
  \begin{lemma}If $\mu$ is a measure that is one-to-one on almost every fiber then $\mu$ can not be the weak-* limit of a sequence of measures $\tilde{\nu}_i$ that are two-to-one on almost every fiber and so that 
  $$\lambda(\{x:diam(supp(\tilde{\nu}_i)_x)>\delta\})>\frac 3 4 $$
  for infinitely many $i$. 
  \end{lemma}
  \begin{proof} There exists $f:[0,1) \to [0,1)$ measurable so that $\mu$ is carried on $\{(x,f(x))\}$. By Lusin's Theorem there exists $\mathcal{K}$ compact with $\lambda(\mathcal{K})>\frac {99}{100}$ so that $f|_{\mathcal{K}}$ is uniformly continuous. 
  Let $s>0$ be so that $d(f(x),f(y))<\frac \delta 8$ for all $x,y \in \mathcal{K}$ with $d(x,y)<s$. 
  Choose an interval $I$ with  $|I|\leq s$, $\lambda(I\cap \mathcal{K})>\frac {99}{100}\lambda(I)$ and 
  \begin{equation}\label{eq:last} \lambda(\{x\in I:diam((\tilde{\nu}_i)_x)>\delta\})>\frac 1 2
  \end{equation} for infinitely many $i$. 
  Let $p=f(x)$ for some $x \in I \cap \mathcal{K}$ and let $g:[0,1) \times [0,1) \to \mathbb{R}$ be a 1-Lipschitz function so that 
  \begin{itemize}
  \item $g|_{I^c\times[0,1)}\equiv 0$
  \item $g|_{I \times B(p,\frac \delta 4)}\equiv 0$
  \item $g(x,y)=\min\{d(x,\partial I), d(y, \partial B(p,\frac \delta 4)),\frac \delta 4\}$ for all $(x,y) \in I \times (B(p,\frac \delta 4))^c$.
  \end{itemize}
  Now $\int g d\sigma\leq .01 |I|\cdot \|g\|_{\sup}\leq  .01|I|\cdot \min\{\frac \delta 4,\frac {|I|} 2\}$. On the other hand if $\tilde{\nu}_i$ satisfies \eqref{eq:last} then on a set of $x \in I$ of measure at least $\frac {|I|}3$ we have one of the two points in $(\tilde{\nu}_i)_x$ is at least $\frac{\delta}2$ away from $p$. A subset of these $x$ of measure at least $\frac{|I|}6$ satisfies $d(x,\partial I)\geq \frac 1 {12} |I|$. So 
  $\int g d\tilde{\nu}_i\geq \frac {|I|} 6  \min\{\frac \delta 4,\frac {|I|} {12}\}$. Since $g$ is 1-Lipschitz it follows that $d_{KR}(\mu,\tilde{\nu}_i)>|I|\min\{|I|(\frac 1 {72}-\frac 1 {200}), \frac \delta {24}-\frac \delta {400}\}$ proving the lemma.
  \end{proof}
  Letting $\tilde{\nu}_i=\frac 1 2 (\nu_i^{(1)}+\nu_i^{(2)})$ and seeing that by \eqref{eq:keep away} they satisfy the condition in the lemma, we see $T$ is not 2-simple. 
\end{proof}

\end{document}